\documentclass[11pt]{article}
\usepackage{amssymb}
\usepackage{amsmath}
\usepackage{amstext}
\usepackage{amsthm}
\usepackage[dvips]{graphicx}
\usepackage{subfigure}
\usepackage{fullpage}
\usepackage{color}

\newcommand{\C}{\mathbb C}
\newcommand{\R}{\mathbb R}

\newcommand{\Z}{\mathbb Z}

\newcommand{\eps}{\varepsilon}

\newcommand{\set}[1]{\left\{#1\right\}}

\newtheorem{theorem}{Theorem}[section]
\newtheorem{lemma}[theorem]{Lemma}
\newtheorem{prop}[theorem]{Proposition}

\theoremstyle{definition}
\newtheorem{definition}[theorem]{Definition}
\newtheorem{rem}[theorem]{Remark}
\newtheorem{rems}[theorem]{Remarks}
\theoremstyle{remark}

\numberwithin{equation}{section}

\begin{document}

\title{On the existence of hylomorphic vortices in the nonlinear Klein-Gordon equation }
\author{J. Bellazzini \thanks{ Universit\`a degli Studi di Sassari, Via
Piandanna 4, Sassari, Italy. e-mail: \texttt{jbellazzini@uniss.it}}  \and V. Benci \thanks{Dipartimento di Matematica, Universit\`a di Pisa, Largo Bruno Pontecorvo 5, Pisa, Italy, and Department of Mathematics, College of Science, King Saud University, Riyadh, 11451, Saudi Arabia. email:\texttt{benci@dma.unipi.it}} \and C. Bonanno \thanks{
Dipartimento di Matematica, Universit\`a di Pisa, Largo Bruno Pontecorvo 5, Pisa, Italy. e-mail: \texttt{bonanno@dm.unipi.it}} \and E. Sinibaldi \thanks{
Istituto Italiano di Tecnologia, Center for Micro-BioRobotics@SSSA,
Viale R. Piaggio 34, Pontedera, Italy. e-mail: \texttt{
edoardo.sinibaldi@iit.it}}}
\date{}
\maketitle

\abstract{In this paper we prove the existence of vortices, namely standing waves with non null angular momentum, for the nonlinear 
Klein-Gordon equation in dimension $N\geq 3$. We show with variational methods that the existence of these kind of solutions, that we have 
called \emph{hylomorphic vortices},
depends on a suitable energy-charge ratio. Our variational approach turns out to be useful for numerical investigations as well.
In particular, some results in dimension $N=2$ are reported, namely exemplificative vortex profiles by varying charge and angular momentum,
together with relevant trends for vortex frequency and energy-charge ratio. The stability problem for hylomorphic vortices is
also addressed. In the absence of conclusive analytical results, vortex evolution is numerically investigated: the obtained
results suggest that, contrarily to solitons with null angular momentum, vortex are unstable. 

\section{Introduction}

Roughly speaking, a \textit{vortex} is a \textit{solitary wave} $\psi $\
with non-vanishing angular momentum $\mathbf{M}\left( \psi \right)$.
A \textit{solitary wave} is a solution of a field equation whose energy
is localized and which preserves this localization in time. The vortices in
the Klein-Gordon equation (KG) are also considered in the Physics literature
with the name of \textit{spinning Q-balls }even if they do not exhibit
spherical symmetry (see e.g. \cite{VW}, \cite{caru}).

We recall also some existence results of solitary waves and vortices for KG:

\begin{itemize}
\item For the case $\mathbf{M}$ $\left( \psi \right) =0,$ we recall the
pioneering paper of Rosen \cite{rosen68} and \cite{Coleman78}, \cite{strauss},
\cite{Beres-Lions}. The
spherically symmetric solitary waves have been called $Q$\textit{-balls} by
Coleman in \cite{Coleman86} and this is the name used in the Physics
literature.

\item Vortices for KG in two space dimensions have been investigated in \cite%
{Kim93}; later also three dimensional vortices for KG have been studied (see 
\cite{Be-Visc}, \cite{bad}, \cite{VW}, \cite{caru},\cite{bebo}).
\end{itemize}

The aim of this paper is twofold. First, we give some existence results in situations not considered in the literature, namely we prove the existence of vortices in the case $N\ge 3$. For the $N=2$ case see \cite{befo09}.
We introduce a new method, based on the energy minimization, which allows to find
vortices with a prescribed charge in a suitable class of functions
(charge is defined in Section \ref{sec:NKG})\footnote{While we were writing this paper, 
a similar existence result appeared in \cite{BR12} for the $N=3$ case under slightly less general assumptions.}.  

Second, by taking advantage of this method,  a numerical investigation is presented.
In particular, some results in dimension $N=2$ are reported, namely exemplificative vortex profiles by varying charge and angular momentum,
together with relevant trends for vortex frequency and energy-charge ratio.  Also, we discuss the
stability of these vortices; such a problem is currently open, but we
present some numerical simulations which suggest that the vortices are unstable. In particular we show that stability cannot be proved by
using the standard methods.  In fact, in order to establish orbital stability, usually one uses the fact that solutions are 
minimizers of the energy functional on the manifold of fixed charge. Unfortunately it turns out that this is not the case, as we show in Theorem \ref{prop-not-min}; actually our vortices are minimizers only in a suitable class of functions.  

This paper is an extension of \cite{bel}, where we addressed the same problem for solitary waves with $\mathbf{M}$ $\left( \psi \right) =0$.

\begin{rem}
In many situations, as in this paper, the existence of stable structures such
as solitary waves and/or vortices is obtained by minimising the energy
(which is defined in Section \ref{sec:NKG})
over a class of configurations of a given charge.
If such a minimizing configuration exists, we may think that there
is a force binding the ``matter'' (see \cite{be09} for details). The
corresponding solitary waves have been called hylomorphic in \cite{bel} (see also 
\cite{befo09}), by merging the Greek words
``\textit{hyle}''=``\textit{matter}'' and
``\textit{morphe}''=``\textit{form}''.
For this reason, the solutions considered in the present paper
can be called ``hylomorphic vortices''.
\end{rem}

The paper is organized as follows: in Section \ref{sec:NKG} we prove the existence of vortices in the case $N\ge 3$ under general assumptions. Moreover the stability problem for the vortices solutions to \eqref{KG} is addressed in subsection \ref{sec:stab-th}. Eventually the numerical investigation is presented in Section \ref{sec:num-vort}. 
\section{The nonlinear Klein-Gordon equation (NKG)} \label{sec:NKG}

Let us consider the nonlinear Klein-Gordon equation (NKG) 
\begin{equation}
\square \psi +W^{\prime }(\psi )=0  \tag{NKG}  \label{KG}
\end{equation}
where $\square =\partial _{t}^{2}-\Delta$, $\psi :\mathbb{R}^{N}\to \mathbb{C}$ with $N\geq 2$, and $W:\mathbb{C}\to \mathbb{R}$ with 
\begin{equation}
W(\psi )=F(|\psi |),\qquad W^{\prime }(\psi )=F^{\prime }(\left\vert \psi \right\vert )\frac{\psi }{\left\vert \psi \right\vert }  \label{www}
\end{equation}
for some smooth function $F:\mathbb{R}^{+}\to \mathbb{R}$. Equation (\ref{KG}) is the Euler-Lagrange equation of the action functional $\mathcal{S}=\int\mathcal{L}dxdt$ with Lagrangian density
$$
\mathcal{L}(\psi,\partial_{t}\psi, \nabla \psi) = \frac{1}{2}\left\vert \partial _{t}\psi \right\vert ^{2}-\frac{1}{2}|\nabla \psi |^{2}-W(\psi). \label{lagra}
$$
It is useful to write $\psi$ in polar form, namely 
$$
\psi(t,x)=u(t,x)e^{iS(t,x)}   \label{polar}
$$
where $u(t,x)\in\mathbb{R}^{+}$ and $S(t,x)\in\mathbb{R}/(2\pi\mathbb{Z})$. If we set $u_{t}=\partial_{t}u$ and $S_{t}= \partial_{t}S$, the state ${\psi}$ is uniquely defined by $(u,u_{t},S_{t}, \nabla u, \nabla S)$. Using these variables, the Lagrangian density $\mathcal{L}$ takes the form 
$$
\mathcal{L}(u,u_{t},S_{t},\nabla u, \nabla S)=\frac{1}{2} \left[
u_{t}^{2}-\left\vert \nabla u\right\vert ^{2}+\left( S_{t}^{2}-|\nabla S|^{2}\right) u^{2}\right] - F(u)
$$
and equation (\ref{KG}) is equivalent to the system
$$
\square u+\left( |\nabla S|^{2}-S_{t}^{2}\right) u+F^{\prime}(u)=0 
$$
$$
\partial_{t}\left( u^{2} S_{t} \right) +\nabla\cdot\left(u^{2} \nabla S \right) =0. 
$$

Noether's theorem states that any invariance for a one-parameter group of
the Lagrangian density implies the existence of an integral of motion, namely of a
quantity on solutions which is preserved with time. Thus equation (\ref{KG}) has for $N=3$ ten integrals: energy, momentum, angular momentum and ergocenter velocity. Moreover, another integral is given by the gauge invariance: charge. Easy computations show that the integrals of motion have the following expressions with respect to the variables $(\psi, \partial_{t} \psi,\nabla \psi)$ or $(u,u_{t},S_{t},\nabla u, \nabla S)$:
\begin{itemize}
\item Energy:
\begin{equation} \label{energy}
\begin{array}{rl}
\mathcal{E} & = \int\left[ \frac{1}{2}\left\vert \partial_{t}\psi\right\vert
^{2}+\frac{1}{2}\left\vert \nabla\psi\right\vert ^{2}+W(\psi)\right] dx \\[0.3cm]
 & = \int\left[ \frac{1}{2}\left( u_{t}\right) ^{2}+\frac {1}{2}\left| \nabla u\right| ^{2}+\frac{1}{2}\left[ S_{t}^{2}+\left| \nabla S\right| ^{2}\right] u^{2}+F(u)\right] dx
 \end{array}
\end{equation}

\item Momentum:  
$$
\mathbf{P}=-\mathrm{Re} \int\partial_{t}\psi\overline{\nabla\psi}\; dx = -\int\left( u_{t}\,\nabla u+S_{t}\,\nabla S\;u^{2}\right) \;dx
$$

\item Angular momentum:
\begin{equation} \label{amomentum}
\mathbf{M}=\mathrm{Re}\int\mathbf{x}\times\nabla\psi\overline
{\partial_{t}\psi}\;dx = \int\left( \mathbf{x}\times\nabla S\ S_{t}\ u^{2}+\mathbf{x}\times\nabla u\,u_{t}\right) \;dx   
\end{equation}

\item Ergocenter velocity:
$$
\mathbf{K}=t\mathbf{P}-\int \mathbf{x}\left[ \frac{1}{2}\left\vert \partial
_{t}\psi \right\vert ^{2}+\frac{1}{2}\left\vert \nabla \psi \right\vert
^{2}+W(\psi )\right] dx 
$$

\item Charge:
\begin{equation} \label{cha}
\mathcal{H}= \mathrm{Im}\int\partial_{t}\psi\overline{\psi}\;dx = \int S_{t}\,u^{2}dx.
\end{equation}
\end{itemize}

A method to obtain vortices or spinning $Q$-balls is to write $x \in \R^N$ as $x=(y,z) \in \R^2
\times \R^{N-2}$ and to look for solutions of the form
\begin{equation} \label{ansatz}
\psi(t,x)=u(x)\, e^{\imath(\ell\, \theta(y)-\omega t)}, \qquad u \geq
0,\ \omega \in \R, \ \ell \in \Z
\end{equation}
where 
$$
\theta(y) = \mathrm{Im}\, \log\left( y_{1}+\imath y_{2}\right) \in \R/2\pi \Z 
$$ 
is the angular variable in the plane $(y_1,y_2)$, for which we set $r^{2}:= y_{1}^{2}+y_{2}^{2}$.
With this ansatz, the NKG reduces to
\begin{equation} \label{nkg-static}
-\Delta u + \left( \ell^2|\nabla \theta|^{2} -\omega^{2} \right) u+ F'(u)=0
\end{equation}
\begin{equation} \label{nkg-static-2}
u \Delta \theta + 2 \nabla u \cdot \nabla \theta=0.
\end{equation}
By definition the function $\theta$ satisfies
$$
\Delta \theta = 0, \qquad \nabla \theta = \left( -\frac{y_{2}}{r^{2}}, \frac{y_{1}}{r^{2}}, 0,\dots,0 \right), \qquad |\nabla \theta| = \frac 1r
$$
and if we assume that 
$$
u(x) = u(r,z)
$$
then (\ref{nkg-static-2}) is satisfied. For solutions of the form \eqref{ansatz}, energy \eqref{energy} and charge \eqref{cha} become
\begin{equation} \label{en-static}
E(u,\omega,\ell)=\int\left[ \frac {1}{2}\left| \nabla u\right| ^{2}+\frac{1}{2}\left[ \omega^{2}+ \frac{\ell^2}{r^2} \right] u^{2}+F(u)\right] dx
\end{equation}
\begin{equation} \label{ch-static}
H(u,\omega)= - \int \omega\, u^{2} \;dx
\end{equation}
Moreover, in the particular case $N=3$, the angular momentum \eqref{amomentum} becomes
\begin{equation} \label{mom-static}
\mathbf{M}(u,\omega,\ell)= \left(0, 0, - \int \ell\, \omega\, u^{2} \;dx \right)=(0,0,\ell H(u,\omega))
\end{equation}
which is a non-vanishing vector if $\ell\not=0$ and $\mathcal{H}(u,\omega) \not= 0$. For this reason solutions (\ref{ansatz}) are called vortices or spinning Q-balls.

\subsection{Existence results for vortices for $N\geq 3$} \label{sec:ex-vort}

Letting $x=(y,z) \in \R^2 \times \R^{N-2}$ and $r:= \sqrt{y_{1}^{2}+y_{2}^{2}}$
we consider the space of cylindrically symmetric functions of the form $u(x) = u(r,z)$. In particular, letting $\mathcal{O}$ be an open subset of $\R^{N-2}$ we consider the Hilbert space 
$$
\tilde H^{1}_{c}(\R^2\times \mathcal{O}) = \overline{\set{\varphi(x) = \varphi(y,z)\, : \, \varphi \in C^{\infty}_{0}((\R^2\setminus \{0\}) \times \mathcal{O})}}^{\|\cdot\|_{\tilde H^{1}_{c}}}
$$
obtained as the closure of cylindrically symmetric functions in $C^{\infty}_{0}((\R^2\setminus \{0\}) \times \mathcal{O})$ with respect to the norm 
$$\|u\|_{\tilde H^{1}_{c}}:=\int_{\R^2\times \mathcal{O}}\, \left[ |\nabla u|^2+ \left( 1 + \frac{1}{r^{2}} \right) |u|^2 \right] dx.$$
We denote by $\| \cdot \|_{H^{1}}$ and $\| \cdot \|_{2}$ the standard norms of the spaces $H^{1}(\R^{N})$ and $L^{2}(\R^{N})$.

We remark that the cylindrical symmetry is sufficient to recover enough compactness thanks to the following lemma which is an immediate consequence of a result of Esteban-Lions \cite{el}
\begin{lemma} \label{compactness}[Compactess Lemma]
Let  $\mathcal{O}$ a bounded open subset of $\R^{N-2}$.
Then the embedding $\tilde H^1_{c}(\R^2 \times \mathcal{O}) \hookrightarrow L^p(\R^N)$ is compact for all $p\in (2,2^{*})$.
\end{lemma}
As discussed above, we consider the cone of non-negative functions $u(x)$ which are cylindrically symmetric and in $\tilde H^{1}_{c}(\R^2\times \mathcal{O})$ with $\mathcal{O} = \R^{N-2}$. Hence we introduce the notation
$$
\tilde H^{1}_{c} := \set{u\in \tilde H^{1}_{c}(\R^2\times \R^{N-2})\ :\ u(x) \ge 0}
$$
We are thus led to prove existence of solutions $(u,\omega, \ell) \in \tilde H^{1}_{c}\times \R \times \Z$ to equation (\ref{nkg-static}) with finite energy (\ref{en-static}). Since we are interested in spinning solutions, we fix $\ell \not= 0$, and look for solutions $(u(\ell), \omega(\ell))$ with $\omega(\ell) \not= 0$. In the following we drop the $\ell$ dependence in notation. Without loss of generality we can restrict ourselves to the case $\omega<0$ (so that charge is positive).

First of all we remark that a variational approach to equation \eqref{nkg-static} in the space $\tilde H^{1}_{c}\times \R \times \Z$ gives the existence of solutions in a weak sense. Then, following for example the argument of Theorem 2.3 in \cite{bebo}, one can prove that any weak solution is also a solution in the sense of distributions. Hence we can look for weak solutions to \eqref{nkg-static}. Our variational approach is based on the following proposition, which is proved by a standard Lagrange multiplier argument (see \cite{bbbm} for an analogous result)

\begin{prop} \label{estr-vinc}
Let $\ell\in \Z \setminus \{0\}$ be fixed. A couple $(\bar u, \bar \omega) \in \tilde H^{1}_{c}\times \R^{-}$ is a weak solution to (\ref{nkg-static}) if and only if it is a critical point for the energy $E(u,\omega,\ell)$ constrained to the manifold $C_{h} := \{(u,\omega) \in \tilde H^{1}_{c}\times \R^-\, : \,  H(u,\omega) = H(\bar u, \bar \omega)  = h \not= 0 \}$ of fixed charge.
\end{prop} 

We have thus obtained a very simple criterion to prove existence of vortices, we just need to prove the existence of constrained critical points for $E$ to $C_{h}$. This problem can be solved by adapting existence results in \cite{bebo} to this case, along the ideas introduced in \cite{bbbm}, or by applying more general results from \cite{befo09}. However we now give a novel proof, which permits to numerically construct the vortices as described in Section \ref{sec:num-vort}.

Let us discuss the assumptions on the nonlinear term $W$ in \eqref{KG}. We write $W$ as 
\begin{equation}
W(\psi)= F(|\psi|) \quad \text{with} \quad F(s)=\frac{m^{2}}{2}s^{2}+N(s), \quad s\ge 0   \label{NN}
\end{equation}
and assume that
\begin{itemize}
\item (W-i)\textbf{(Positivity}) $F(s)\ge 0$;

\item (W-ii)\textbf{(Nondegeneracy}) $F(s)$ is $C^{2}$ with $F(0)=F^{\prime}(0)=0$, $F^{\prime\prime}(0)=m^{2} > 0$;

\item (W-iii)\textbf{(Hylomorphy}) there exists $s_{0}>0$ such that $N(s_{0})<0$;

\item (W-iv)\textbf{(Growth condition}) at least one of the following holds:

\begin{itemize}
\item (a) there are constants $a,b>0$ and $2<p\le q<2N/(N-2)$ such that for any $s>0$
$$
|N^{\prime}(s)| \le as^{p-1}+bs^{q-1}. 
$$

\item (b) there exists $s_{1}>s_{0}$ such that $N^{\prime}(s_{1})\ge 0$.
\end{itemize}
\end{itemize}

\begin{rems}
We make some comments on assumptions (W-i), (W-ii), (W-iii), (W-iiii).

(W-i) By \eqref{en-static}, (W-i) implies that the energy is positive.

(W-ii) The necessary condition for the existence of solitary waves for \eqref{KG} is $F^{\prime\prime}(0)\ge 0$. Results for the null-mass case, $F^{\prime\prime}(0)=0$, are obtained in e.g. \cite{Beres-Lions} and \cite{BBR07}. Here we consider only the positive-mass case.

(W-iii) This is the crucial assumption which characterizes the potentials
which might produce concentrated solutions. As we will see, this assumption
permits to have states $\psi$ with hylomorphy ratio $\Lambda\left(\psi\right) <m$ (see \eqref{hylo} below).

(W-iv)(a) This assumption contains the usual growth condition at infinity
which guarantees the $C^{1}$ regularity of the functional. If we assume alternatively (W-iv)(b), the growth condition (W-iv)(a) can
be recovered by using standard tricks (see \cite{bbbm}).
\end{rems}

Our main existence results is the following
\begin{theorem}\label{stand}
Under assumptions (W-i)-(W-iv) and for any fixed $\ell\in \Z\setminus \{0\}$, there exists $h_{0}\in \R^{+}$ such that for all $h\ge h_{0}$ the nonlinear Klein-Gordon equation \eqref{KG} admits a vortex solution $\psi(t,x)$ of the form \eqref{ansatz} with finite energy \eqref{energy}, charge \eqref{cha} $\mathcal{H} = h$ and, for $N=3$, non-vanishing angular momentum \eqref{amomentum}.
\end{theorem}

For fixed $\ell \in \Z \setminus \{0\}$, it is sufficient by Proposition \ref{estr-vinc} to show that the energy $E(u,\omega,\ell)$ has a point of minimum on the manifold $C_{h}$ for $h$ big enough. However, it is possible and useful for numerical aims to reduce the problem to the minimization of a one-variable functional. Namely, for fixed $\ell \in \Z \setminus \{0\}$, for all the couples $(u, \omega) \in C_h$ the energy functional $E(u, \omega, \ell)$ can be rewritten as
\begin{equation} \label{j-def}
J_{h}(u,\ell):= E(u,\omega(u,h),\ell) = \int\left[ \frac {1}{2}\left| \nabla u\right| ^{2}+\frac{1}{2}\frac{\ell^2}{r^2}u^{2}+N(u) \right] dx +\frac 12\left( \frac{h^2}{\|u\|_2^2}+m^2\| u\|_2^2\right)
\end{equation}
where $\ell$ and $h$ are parameters. The existence of a minimum for $E$ on $C_{h}$ is then equivalent to the existence of a minimum of $J_{h}$ on $\tilde H^{1}_{c} \setminus \{0\}$. 

Finally we introduce a fundamental tool in our variational approach, the quantity
\begin{equation}\label{hylo}
\Lambda(u, \omega):=\frac{E( u, \omega,\ell)}{H( u,  \omega)}
\end{equation}
that we called the \textit{hylomorphy ratio} (see \cite{bel}).

Theorem \ref{stand} follows from the two following results

\begin{theorem}\label{mezzo-1}
Under assumptions (W-i)-(W-iv) and for any fixed $\ell\in \Z\setminus \{0\}$, if there exists a  couple $(u, \omega) \in C_{h}$ such that $\Lambda(u, \omega)<m$ then $J_{h}$ has a point of minimum on $\tilde H^{1}_{c} \setminus \{0\}$.
\end{theorem}

\begin{theorem}\label{mezzo-2}
Under assumptions (W-i)-(W-iv) and for any fixed $\ell\in \Z\setminus \{0\}$, there exists $h_{0}\in \R^{+}$ such that for all $h\ge h_{0}$
\begin{equation}\label{mainhylo2}
\inf_{(u,\omega)\in C_{h}} \Lambda(u,\omega) < m
\end{equation}
\end{theorem}

The existence of hylomorphic vortices for all charges can be obtained by a stronger version of assumption (W-iii). Let us consider the condition
\begin{itemize}
\item (W-v) \textbf{(Behaviour at $s=0$)} $N(s)\le -s^{2+\eps}$  with  $ 0<\eps<\frac{4}{N}$ for $s\in \R^{+}$ small enough ($N(s)$ is defined in \ref{NN}).
\end{itemize}

\begin{theorem}\label{stand2}
Under assumptions (W-i)-(W-ii)-(W-iv)-(W-v) and for any fixed $\ell\in \Z\setminus \{0\}$ and for all $h \in \R^+$ the nonlinear Klein-Gordon equation \eqref{KG} admits a vortex solution $\psi(t,x)$ of the form \eqref{ansatz} with finite energy \eqref{energy}, charge \eqref{cha} $\mathcal{H} = h$ and, for $N=3$, non-vanishing angular momentum \eqref{amomentum}.
\end{theorem}

We first give the proof of Theorem \ref{mezzo-1}, which is based on some preliminary results.

\begin{lemma}[A priori estimates]\label{apriori}
Assume conditions (W-i), (W-ii) and (W-iv)(a), and let 
$$m(\ell, h):=\inf_{u \in \tilde H^{1}_{c} \setminus \{0\}} J_{h}(u, \ell)\, .$$
Then there exist $K_1, K_2>0$ such that if $(u_n)\subset \tilde H^{1}_{c}$ satisfies
$$J_{h}(u_n, \ell)\rightarrow m(\ell, h)$$
then $K_1 \le \|u_{n}\|_{2} \le \| u_n\|_{H^1}\le \| u_{n} \|_{\tilde H^{1}_{c}}\le K_2$.
\end{lemma}
\begin{proof}
First of all we can rewrite $J_{h}(u,\ell)$ defined in \eqref{j-def} as
\begin{equation}\label{j-new}
J_{h}(u,\ell) = \int\left[ \frac {1}{2}\left| \nabla u\right| ^{2}+\frac{1}{2}\frac{\ell^2}{r^2}u^{2}+F(u) \right] dx +\frac 12 \frac{h^2}{\|u\|_2^2}
\end{equation}
which is thanks to (W-i) a sum of non-negative terms. Hence any minimizing sequence $(u_{n})$ satisfies
$$
\| \nabla u_{n} \|_{2} \le \alpha_{1},\quad \int \frac{u_{n}^2}{r^2}\, dx \le \alpha_{2}, \quad \| u_{n} \|_{2} \ge \alpha_{3}
$$
for some positive constants $\alpha_{1}, \alpha_{2}, \alpha_{3}$.

In order to prove that $\|u_n\|_{\tilde H^1_{c}}\leq K_2$ it remains to show that 
$\|u_n\|_2\le \alpha_{4}$ for some positive $\alpha_{4}$. Here we follow an argument from \cite{bbbm} which we include for completeness. From (W-ii) it follows that
\begin{equation}  \label{vicino-zero}
\exists\, \delta>0 \ \exists\, c_1>0, \text{ such that } F(s) \geq c_1
s^2 \text { for } 0\le s \le \delta.
\end{equation}
Let us assume by contradiction that
$$
\| u_n \|_{2} \rightarrow \infty.
$$
By \eqref{j-new} $\int F(u_n)\,dx$ is bounded and by  (W-i) and (\ref{vicino-zero})
\begin{equation}  \label{for}
\int F(u_n)\, dx \ge \int_{0 \le u_n\le \delta} F(u_n)\, dx \ge c_1\,
\int_{0 \le u_n \le \delta} u_n^2\, dx .
\end{equation}
On the other hand
$$
\|u_{n}\|_{2}^{2} = \int_{0\le u_n\le \delta} u_n^2\, dx +\int_{u_n\ge \delta} u_n^2\, dx
\rightarrow \infty,
$$
thus we have by (\ref{for})
$$
\int_{u_n\ge \delta} u_n^2\, dx \rightarrow \infty.
$$
This drives to a contradiction since for $2^{*}= \frac{2N}{N-2}$
$$
\frac{1}{\delta^{2^*-2}} \ \int_{u_n\ge \delta} u_n^{2^*}dx \ge
\int_{u_n\ge \delta} u_n^2\, dx
$$
and by the Sobolev embedding theorem
$$
\int_{u_n\ge \delta} u_n^{2^*}dx \le \int u_n^{2^*}dx \le c_{2}\, \| \nabla u_{n} \|_{2}^{\frac{2^{*}}{2}} \le c_{2}\, \alpha_{1}
$$
where $c_{2}$ is the Sobolev embedding constant.
\end{proof}

\begin{lemma}\label{nonvanish}
Let $u_n$ be a bounded sequence in $\tilde H^{1}_{c}$ such that 
$$\|u_n\|_{L^{q}}\ge \delta >0 \text{ for some } q \in \left(2, \frac{2N}{N-2} \right).$$
Then, up to subsequence, there exist a sequence $(\tau_n) \subset \R^{N-2}$ and $\bar u \in \tilde H^1_c$, $\bar u\not\equiv 0$, such that $u_n(\cdot+\tau_n)$ converges weakly to $\bar u$.
\end{lemma}
\begin{proof}
See Lemma 3.6 in \cite{bebo}.
\end{proof}

Finally we need the following compactness result shown in \cite{bebo} for a constrained minimization problem of functionals on the $L^{2}$ balls
\begin{lemma}[Minimization problem on $L^2$ constraint \cite{bebo}]\label{jc}
Let us consider the following minimization problem
$$G_{\rho}:=\inf_{v \in B_{\rho}}G(v)$$
where
$$G(v):=\int\left[ \frac {1}{2}\left| \nabla v\right| ^{2}+\frac{1}{2}\frac{\ell^2}{r^2}v^{2}+N(v) \right] dx,$$
and $B_{\rho}=\{v \in \tilde H^1_c \, : \, \|v\|^2_2=\rho\}$. Under assumptions (W-ii), (W-iii) and (W-iv), if there exist $v_0\in B_{\rho}$ such that $G (v_0)<0$ then 
for any sequence $v_n$ such that 
$$
G(v_n) \rightarrow G_{\rho}
$$
there exist, up to subsequence, a sequence $(\tau_n) \subset \R^{N-2}$ and a $\bar v \in \tilde H^1_c$, $\bar v \not\equiv 0$, such that $v_n(\cdot+\tau_n)$ converges to $\bar v$ in the $\tilde H^{1}_{c}$ norm.
\end{lemma}

\begin{proof}[Proof of Theorem \ref{mezzo-1}]
First of all, fixed $\ell \in \Z \setminus \{0\}$, we can rewrite the hylomorphy ratio \eqref{hylo} as a function of the single variable $u$, that is
\begin{equation}\label{hylo-j}
\Lambda(u,\ell) = \frac{J_{h}(u,\ell)}{h} = \frac 1h \int\left[ \frac {1}{2}\left| \nabla u\right| ^{2}+\frac{1}{2}\frac{\ell^2}{r^2}u^{2}+N(u) \right] dx + \frac 12\left( \frac{h}{\|u\|_2^2}+\frac{m^2}{h} \| u\|_2^2\right)
\end{equation}
and remark that for all $u \in \tilde H^1_c$ and all $h>0$
\begin{equation}\label{noticina}
\frac 12\left( \frac{h}{\|u\|_2^2}+\frac{m^2}{h} \| u\|_2^2\right) \ge m.
\end{equation}
Therefore, if there exists $(u,\omega) \in C_{h}$ such that $\Lambda(u,\omega) <m$, by \eqref{hylo-j} there exists $u\in \tilde H^{1}_{c}$ such that $J_{h}(u,\ell) < mh$. Hence, thanks to \eqref{noticina}, for any minimizing sequence $(u_n)\subset \tilde H^{1}_{c}$ such that $J_{h}(u_n,\ell) \rightarrow m(\ell,h)$ there exist $\bar n$ and $\epsilon_0<0$ such that for all $n\ge \bar n$
$$G(u_n):=\int\left[ \frac {1}{2}\left| \nabla u_n\right| ^{2}+\frac{1}{2}\frac{\ell^2}{r^2}u_n^{2}+N(u_n) \right] dx\le \epsilon_0<0. $$
This fact and the growth condition (W-iv) guarantee that there exist $\delta>0$ and $q \in \left(2, \frac{2N}{N-2}\right)$ such that $\|u_n\|_{L^{q}}>\delta$ for $n\ge \bar n$. Hence we can apply Lemmas \ref{apriori} and \ref{nonvanish} to conclude that up to subsequence there exists a sequence $(\tau_n) \subset \R^{N-2}$ such that
\begin{equation}  \label{converg}
u_n(\cdot + \tau_n) \rightharpoonup u_0\neq 0  \text{ weakly in } \tilde H^1_c
\end{equation}
with
\begin{equation}  \label{rho}
\int u_n^2\, dx \to \rho>0
\end{equation}
by Lemma \ref{apriori}. We claim that
\begin{equation}  \label{inf-uguali}
m(\ell,h)= G_{\rho} + \frac  {m^2}{2} \, \rho +
\frac{h^2}{2\rho}
\end{equation}
where $\rho$ is defined in (\ref{rho}) and $G_{\rho}:= \inf_{B_{\rho}} G$ (see Lemma \ref{jc}). In order to prove \eqref{inf-uguali} we take $v_n=\frac{\sqrt{\rho}}{\|u_n\|_{2}}u_n \in B_\rho$ and show
that
\begin{equation}  \label{convpalla}
G(v_n)-G(u_n)\rightarrow 0.
\end{equation}
Indeed,
{\small
\begin{equation*}
\begin{array}{l}
|G(v_n)-G(u_n)| = |(\frac{\rho}{\|u_n\|^2_{2}}-1) \frac 12 \int [ |\nabla u_n|^2+\frac{\ell^2}{r^2}u^{2} ] dx +
\int N(\frac{\sqrt{\rho}}{\|u_n\|_{2}}u_n)-N(u_n)dx| \le \\[0.3cm]
\le |(\frac{\rho}{\|u_n\|^2_{2}}-1)| \frac 12 \int [ |\nabla u_n|^2 +\frac{\ell^2}{r^2}u^{2}] dx+ |(\frac{\sqrt{
\rho}}{\|u_n\|_{2}}-1)| \int |N^{\prime}(\theta\frac{\sqrt{\rho}}{
\|u_n\|_{2}}u_n+(1-\theta)u_n)| u_n\, dx \le \\[0.3cm]
\le |(\frac{\rho}{\|u_n\|^2_{2}}-1)| \frac{\ell^{2}}{2} \| u_{n} \|_{\tilde H^{1}_{c}}^{2} + |(\frac{\sqrt{
\rho}}{\|u_n\|_{2}}-1)| \int |N^{\prime}(\theta\frac{\sqrt{\rho}}{
\|u_n\|_{2}}u_n+(1-\theta)u_n)| u_n\, dx
\end{array}
\end{equation*}}
\noindent for some function $\theta: \R^{N}\to (0,1)$. Moreover, $\| u_{n} \|_{\tilde H^{1}_{c}}^{2}$ is bounded by Lemma \ref{apriori}, and applying (W-iv)(a) we get
{\small
\begin{equation*}
\begin{array}{l}
\int |N^{\prime}(\theta\frac{\sqrt{\rho}}{\|u_n\|_{2}}u_n+ (1-\theta) u_n)|
u_n\, dx \le \\[0.3cm]
\le a\int|\left(\theta\frac{\sqrt{\rho}}{\|u_n\|_{2}}+(1-\theta)
\right)^{p-1}|\, |u_n|^p dx + b \int |\left(\theta\frac{\sqrt{\rho}}{
\|u_n\|_{2}}+(1-\theta) \right)^{q-1}|\, |u_n|^{q} dx 
\end{array}
\end{equation*}} 
\noindent and the right-hand side is bounded thanks to the Sobolev embedding theorem. Hence (\ref{convpalla}) follows since $\frac{\rho}{\|u_n\|^2_{2}} \rightarrow 1$. Hence it follows that
$$
G(v_n) \rightarrow  m(\ell,h)- \frac{m^{2}}{2}\, \rho-\frac{h^2}{2\rho}.
$$
and
$$
\inf_{B_{\rho}} G  \le m(\ell, h) - \frac{m^{2}}{2}\, \rho-\frac{ h^2}{2\rho}.
$$
The opposite inequality holds since 
$$
\inf_{u\in B_{\rho}} G(u) = \inf_{u\in B_{\rho}} \left( J_{h}(u,\ell) - \frac 12 \frac{h^{2}}{\|u\|_2^2}+\frac{m^2}{2} \| u\|_2^2 \right) = \inf_{u\in B_{\rho}}  J_{h}(u,\ell) - \frac{m^{2}}{2}\, \rho-\frac{h^2}{2\rho} \ge
$$
$$
\ge m(\ell,h) - \frac{m^{2}}{2}\, \rho-\frac{h^2}{2\rho}
$$
where in the last inequality we used the fact that $m(\ell,h)$ is computed on a set containing $B_{\rho}$.

Hence we have proved that any minimizing sequence $(u_{n})\subset \tilde H^{1}_{c}$ for $J_{h}$, with $m(\ell,h)< mh$, satisfies \eqref{converg} and \eqref{rho}, and we find a sequence $v_{n} \in B_{\rho}$ which is minimizing for a functional $G$ satisfying the assumptions of Lemma \ref{jc}. Indeed from \eqref{inf-uguali} it follows that 
$$G_{\rho} + \frac  {m^2}{2} \, \rho + \frac{h^2}{2\rho} < mh$$
whence that $G_{\rho}<0$. And from \eqref{convpalla} it follows that $G(v_{n}) \to G_{\rho}$. Hence, applying Lemma \ref{jc}, it follows that, up to subsequence, there exist a sequence $(\tau_n) \subset \R^{N-2}$ and a $\bar v \in \tilde H^1_c$, $\bar v \not\equiv 0$, such that $v_n(\cdot+\tau_n)$ converges to $\bar v$ in the $\tilde H^{1}_{c}$ norm. Hence
$$
u_{n}(\cdot+\tau_n) = \frac{\| u_{n}\|_{2}}{\sqrt{\rho}} \, v_{n}(\cdot+\tau_n) \to \bar v \quad \text{in the $\tilde H^{1}_{c}$ norm}
$$
since $\|u_{n}\|_{2}\to \sqrt{\rho}$. Theorem \ref{mezzo-1} is proved.
\end{proof}

\begin{proof}[Proof of Theorem \ref{mezzo-2}]
We start the proof by defining 
\begin{equation}\label{alfa0}
\alpha _{0}:=\inf\limits_{s\in \R^{+} }\ \frac{F(s)}{\frac{1}{2}s^{2}}.
\end{equation}
Thanks to (W-iii) we have
$$\alpha_0<m^2.$$
The hylomorphy ratio \eqref{hylo} can be written also as
$$
\Lambda(u, \omega)=\frac{E( u, \omega)}{H( u,  \omega)}= - \frac 1 2 \, \omega
- \frac{1}{2\omega} \, \alpha(u)
$$
where
$$
\alpha(u):= \frac{\int \left( \frac{1}{2} |\nabla u|^2 +\frac 12\frac{\ell^2}{r^2}|u|^2+ F(u) \right)dx}{
\int \frac 1 2 \, u^2 \ dx}.
$$
Hence, for any fixed $u$, we have
$$
\inf\limits_{\omega \in \R^-} \, \Lambda( u,\omega) = \Lambda\left(u, -\sqrt{\alpha(u)}\right) = \sqrt{\alpha(u)}.
$$
Hence it is sufficient to show that
$$
\inf\limits_{u\in \tilde H^1_c} \, \alpha(u) \le \alpha_0<m^2.
$$
By definition of $\alpha_{0}$ \eqref{alfa0}, for any fixed $\epsilon>0$ there exists $s_0\in \R^+$ such that
$$
\frac{F(s_0)}{\frac{1}{2}s_0^{2}}<\alpha_0+\frac{\epsilon}{2}
$$
Then let us consider the sequence of functions $u_n(x) = u_n(|y|,z) = f(|z|)\, v_n(|y|)$
with
$$
v_n(|y|) = \left\{
\begin{array}{ll}
0 & \text{for $|y| \le R_n-1$} \\[0.1cm]
s_0 (|y|-R_n+1) & \text{for $R_n-1 \le |y| \le R_n$} \\[0.1cm]
s_0 & \text{for $R_n \le |y| \le 2 R_n $} \\[0.1cm]
s_0 (2 R_n-|y|+1) & \text{for $2 R_n \le |y| \le 2 R_n+1$} \\[0.1cm]
0 & \text{for $|y| \ge 2 R_n+1$}
\end{array} \right.
$$
$$
f(|z|) = \left\{
\begin{array}{ll}
1 & \text{for $0\le |z| \le 1$} \\[0.1cm]
2-|z| & \text{for $1 \le |z| \le 2$} \\[0.1cm]
0 & \text{for $|z| \ge 2$}
\end{array}
\right.
$$
and assume $R_n \to \infty$. 
Then $u_n \in \tilde H^1_c$ and
$$
\int\, |\nabla u_n|^2\, dx = O(R_n)
$$
$$
\int\, \frac{\ell^2}{r^2}u_n^2\, dx = \left( \int_{R_n}^{2 R_n}\,  \frac{\ell^2}{r^2} \,
r \, s_0^2\, dr \right) \, \left( \int_{|z|\le 2}\, f(z) dz
\right) + o(\ln(R_n)) = O(\ln(R_n))
$$
$$
\int\, F(u_n)\, dx = \left( \int_{R_n}^{2 R_n}\, r\,
F(s_0)\, dr \right) \, \left( \int_{|z|\le 2}\, f(z) dz \right) +
o(R_n) \leq 
$$
$$\leq \left( \int_{R_n}^{2 R_n}\, r\,
\frac{1}{2}s_0^2(\alpha_0+\frac{\epsilon}{2})\, dr \right) \, \left( \int_{|z|\le 2}\, f(z) dz \right) +
o(R_n)=const (\alpha_0+\frac{\epsilon}{2}) s_{0}^{2} R_{n}^{2}  +o(R_n)
$$
$$\int \frac 12 |u_n|^2 dx=  \left( \int_{R_n}^{2 R_n}\, r\,
\frac 12 s_0^2\, dr \right) \, \left( \int_{|z|\le 2}\, f(z) dz \right) +
o(R_n) = const\, s_{0}^{2} R_{n}^{2}  +o(R_n)$$
It follows that $\alpha(u_n)<\alpha_0+\epsilon$ for sufficiently large $n$. Moreover
$$H\left(u_{n}, -\sqrt{\alpha(u_{n})}\right) \approx R_{n}^{2}.$$
We have thus proved that there exists $h_{0}\in \R^{+}$ such that
$$
\inf_{(u,\omega)\in C_{h_{0}}}\, \Lambda(u,\omega) < m.
$$
Let now $h>h_{0}$ and $(\bar u,\bar \omega) \in C_{h_{0}}$ such that $\Lambda(\bar u,\bar \omega) < m$. We want to show that there exists $(u,\omega) \in C_{h}$ such that $\Lambda(u,\omega) < m$. Recalling that $x=(y,z) \in \R^{2}\times \R^{N-2}$, let us define for $\lambda \in \R$
$$
u_{\lambda}(y,z) = \bar u \left( \frac{y}{\lambda}, z \right),\quad \omega = \bar \omega
$$
Then $(u_{\lambda}, \omega) \in C_{h}$ if 
$$
-\omega\,  \| u_{\lambda} \|_{2}^{2} = -\bar \omega \, \lambda^{2} \| \bar u \|_{2}^{2} = -\bar \omega\, \lambda^{2} \frac{h_{0}}{-\bar \omega} = h
$$
that is, if $\lambda^{2} = \frac{h}{h_{0}}>1$. Moreover, for the single terms appearing in $\Lambda$ it holds
$$
\int |\nabla u_{\lambda}|^{2} dx = \int \left( |\nabla_{y} u_{\lambda}|^{2} + |\nabla_{z} u_{\lambda}|^{2} \right) dx = \int \left( |\nabla_{y} \bar u|^{2} + \lambda^{-2} |\nabla_{z} \bar u|^{2} \right) dx \le \int |\nabla \bar u|^{2} dx\, ,
$$
$$
\int \frac{\ell^{2}}{r^{2}}\, u_{\lambda}^{2} \, dx = \int \frac{\ell^{2}}{r^{2}}\, \bar u^{2} \, dx
$$
since $r^{2} = |y|^{2}$, and
$$
\int N(u_{\lambda})\, dx = \lambda^{2} \int N(\bar u) \, dx\, .
$$
Hence using also $h>h_{0}$ it follows
$$
\Lambda(u_{\lambda}, \omega) = \frac 1h \int\left[ \frac {1}{2}\left| \nabla u_{\lambda}\right| ^{2}+\frac{1}{2}\frac{\ell^2}{r^2}u_{\lambda}^{2}+N(u_{\lambda}) \right] dx + \frac 12\left( \frac{h}{\|u_{\lambda}\|_2^2}+\frac{m^2}{h} \| u_{\lambda}\|_2^2\right) \le
$$
$$
\le \frac{1}{h_{0}} \int \frac {1}{2} \left[ \left| \nabla \bar u \right| ^{2} + \frac{\ell^{2}}{r^{2}}\, \bar u^{2} \right] \, dx + \frac{\lambda^{2}}{h} \int N(\bar u) \, dx + \frac 12\left( -\omega - \frac{m^2}{\omega}\right) =
$$
$$
= \frac{1}{h_{0}} \int\left[ \frac {1}{2}\left| \nabla \bar u \right| ^{2}+\frac{1}{2}\frac{\ell^2}{r^2}\bar u^{2}+N(\bar u) \right] dx + \frac 12\left( \frac{h_{0}}{\|\bar u \|_2^2}+\frac{m^2}{h_{0}} \| \bar u\|_2^2\right) = \Lambda(\bar u , \bar \omega) < m
$$
and the proof is finished.
\end{proof}
\begin{proof}[Proof of Theorem \ref{stand2}]
Arguing as in the proof of Theorem \ref{stand} we only need to show that
for any fixed $\ell\in \Z\setminus \{0\}$ and for all $h\in\R^+$
\begin{equation}\label{mainhylo3}
\inf_{(u,\omega)\in C_{h}} \Lambda(u,\omega) < m
\end{equation}
Choosing $\omega=-m$ we have
\begin{equation*}
\inf_{(u,\omega)\in C_{h}} \Lambda(u,\omega) \leq \inf_{(u,-m)\in C_{h}}\Lambda(u,-m).
\end{equation*}
Notice that $\Lambda(u,-m)<m$ if
$$\int \left( \frac{1}{2} |\nabla u|^2 +\frac 12\frac{\ell^2}{r^2}|u|^2+ N(u) \right)dx<0$$
Hence it is sufficient to show that for any $\rho>0$
$$\inf_{u\in \tilde H^1_{c,\rho}} \int \left( \frac{1}{2} |\nabla u|^2 +\frac 12\frac{\ell^2}{r^2}|u|^2+ N(u) \right)dx<0 $$
where
$$\tilde H^1_{c,\rho}:=\{u \in \tilde H^1_c \text{ s.t }||u||_2^2=\rho \}.$$
Now let us consider the sequence of functions $u_n(x) = u_n(|y|,z) = f(|z|)\, v_n(|y|)$
with
$$
v_n(|y|) = \left\{
\begin{array}{ll}
0 & \text{for $|y| \le R$} \\[0.1cm]
-s_n+\frac{s_n}{R_n}|y| & \text{for $R_n \le |y| \le 2R_n$} \\[0.1cm]
s_n  & \text{for $2R_n \le |y| \le 4 R_n $} \\[0.1cm]
5s_n-\frac{s_n}{R_n}|y| & \text{for $4 R_n \le |y| \le 5 R_n$} \\[0.1cm]
0 & \text{for $|y| \ge 5 R_n$}
\end{array} \right.
$$
$$
f(|z|) = \left\{
\begin{array}{ll}
1 & \text{for $0\le |z| \le R_n$} \\[0.1cm]
1-\frac{1}{R_n}(|z|-R_n) & \text{for $R_n \le |z| \le 2|R_n|$} \\[0.1cm]
0 & \text{for $|z| \ge 2R_n$}
\end{array}
\right.
$$
and assume $R_n \to \infty$ and $s_n \rightarrow 0$ such that $\int |u_n|^2dx=\rho$. Notice that
$$\int |u_n|^2dx=\left(\int_{\R^2}v_n(|y|)^2 dy\right)\left(\int_{\R^{N-2}}f(|z|)^2dz\right)$$
such that a simple scaling analysis shows that $\lim_{n \rightarrow \infty}s_n^2R_n^N=\gamma>0$. Moreover we get
$$
\int_{\R^N} |\nabla u_n|^2dx=\left(\int_{\R^2}|\nabla v_n|^2dy \right)\left(\int_{\R^{N-2}}f(|z|)^2dz\right)+\left(\int_{\R^2} v_n(|y|)^2dy \right)\left(\int_{\R^{N-2}}|\nabla f|^2dz\right)=O(R_n^{-2})
$$
$$
\int_{\R^N} N(u_n)dx \leq -\int_{\R^d}|u_n|^{2+\epsilon}dx=-\left(\int_{\R^2}v_n(|y|)^{2+\epsilon}dy \right)\left(\int_{\R^{N-2}} f(|z|)^{2+\epsilon}dz\right)=-O(R_n^{-\frac{1}{2}\epsilon N})
$$
$$\int\, \frac{\ell^2}{r^2}u_n^2\, dx \leq \frac{\ell^2}{R_n^2}\left( \int_{R_2} v_n(|y|)^2 dy\right)\left( \int_{\R^{N-2}}\, f(|z|)^2 dz\right)=O(R_n^{-4})
$$
The previous estimates imply for $R_n\rightarrow \infty$ and $0<\epsilon<\frac{4}{N}$ that
$$\int \left( \frac{1}{2} |\nabla u_n|^2 +\frac 12\frac{\ell^2}{r^2}|u_n|^2+ N(u_n) \right)dx\rightarrow 0^-$$
i.e that
$$\inf_{(u,\omega)\in C_{h}} \Lambda(u,\omega) \leq \inf_{(u,-m)\in C_{h}}\Lambda(u,-m)<m.$$
\end{proof}
\subsection{The stability problem} \label{sec:stab-th}

We now discuss the possibility that the vortices solutions to \eqref{KG} found in the previous section are \textit{orbitally stable}, in the sense of the following definition (see e.g. \cite{gss} and \cite{bbbm}). For simplicity of notations we consider the case $N=3$, so that we can use expressions \eqref{amomentum} and \eqref{mom-static} for the angular momentum. However the same argument holds in any dimension $N\ge 3$.

\begin{definition}\label{def-os}
A vortex solution $\bar \psi(t,x) = u(x) e^{\imath (\ell \vartheta (y) - \omega t)}$ is called orbitally stable if the set
$$
\Gamma(u,\ell,\omega) := \set{u(x+\chi) e^{\imath (\ell \vartheta (y+\xi) - \omega t-\theta)}\ :\ \chi=(\xi,\zeta)  \in \R^{2}\times \R,\, \theta \in \R}
$$
is stable in the $H^{1}\times L^{2}$ norm, i.e. for any $\varepsilon >0$ there exists $\delta >0$ such that $d(\psi(0,\cdot), \Gamma) < \delta$ implies $d(\psi(t,\cdot), \Gamma) < \varepsilon$ for all $t\in \R$, where
\begin{equation}
d( \psi (t,\cdot), \bar \psi(t,\cdot)) := \| \psi(t, \cdot) - u(\cdot) e^{\imath (\ell \vartheta (\cdot) - \omega t)} \|_{H^{1}} + \| \partial_{t} \psi(t, \cdot) - (-\imath \omega) u(\cdot) e^{\imath (\ell \vartheta (\cdot) - \omega t)} \|_{L^{2}}
\label{os_continuo}
\end{equation}
\end{definition}

To establish orbital stability there are essentially two methods used in literature, one based on the Lions's Compactness-Concentration Lemma and one introduced by Shatah (see \cite{gss}). However both methods use the idea that stability can be obtained 
for solutions that are non-degenerate points of global minimum of the energy functional constrained on the manifold of fixed charge. 
Hence we could investigate whether our solutions $\bar \psi(t,x) = u(x) e^{\imath (\ell \vartheta (y) - \omega t)}$ are points of minimum 
for the energy $\mathcal E$ constrained on the manifold of fixed charge ${\mathcal H}(\bar \psi)$ and fixed angular 
momentum ${\mathbf M}(\bar \psi)$. Unfortunately it turns out that this is not the case, as we show in the main result of this section 
Theorem \ref{prop-not-min}. However it remains open the possibility that these solutions are non-degenerate points of local minimum, as it happens for solutions found in \cite{bona} for \eqref{KG} for solitons with vanishing angular momentum. This weaker property would be sufficient for orbital stability, but this is still an open question.

Let us recall the expressions ${\mathcal E}$ \eqref{energy} and $E$ \eqref{en-static} for the energy on a general function $\psi \in H^1(\R^3,\C)$ and on the vortices solutions satisfying the ansatz \eqref{ansatz} respectively. The solutions $\bar \psi$ in the previous section have been found as points of global minimum for the energy $E$ on the manifold of couples $(u,\omega)$ with fixed charge. We now show that these solutions do not attain the global minimum of the energy $\mathcal E$.

\begin{theorem} \label{prop-not-min}
For any fixed $\ell\in \Z \setminus \{0\}$, let $\bar \psi(t,x) = \bar u(x) e^{\imath (\ell \vartheta (y) - \bar \omega t)}$ be a solution of \eqref{KG}, with $h = {\mathcal H}(\bar \psi)$ and ${\mathbf m}= {\mathbf M}(\bar \psi)$, found as in the proof of Theorem \ref{stand}, that is $E(\bar u,\bar \omega)$ is the minimum of $E$ on the manifold $C_h$. Then
$$
\inf\limits_{{\mathcal H}=h,\, {\mathbf M}={\mathbf m}}\, {\mathcal E} = \inf\limits_{{\mathcal H}=h}\, {\mathcal E} = \inf\limits_{(u,\omega) \in C_{h}}\, E(u,\omega,0) < {\mathcal E}(\bar \psi)
$$
\end{theorem}

\begin{proof}
\textit{Step 1.} ${\mathcal E}(\bar \psi) > \inf\limits_{(u,\omega) \in C_{h}}\, E(u,\omega,0)$

First of all it holds
$$
{\mathcal E}(\bar \psi) = E(\bar u, \bar \omega, \ell) \ge \inf\limits_{(u,\omega) \in C_{h}}\, E(u,\omega,\ell)
$$
which is given by construction. Moreover we show that
$$
e_{r,\ell}:= \inf\limits_{(u,\omega) \in C_{h}}\, E(u,\omega,\ell) > \inf\limits_{(u,\omega) \in C_{h}}\, E(u,\omega,0)=: e
$$
Let $(u_{n}, \omega_{n}) \in C_{h}$, with $u_{n}\in \tilde H^{1}_{c}$, be a minimising sequence for $e_{r,\ell}$. By Lemma \ref{apriori} there exist $K_1, K_2 >0$ such that
\begin{equation} \label{s1}
K_{1} \le \| u_{n} \|_{L^{2}}\le \| u_{n} \|_{H^1}\le \| u_n \|_{\tilde H^1_c} \le K_{2}
\end{equation}
If by contradiction $e_{r,\ell} = e$ then, using the notation $B_{R}(0):= \set{ y\in \R^{2}\ :\ r< R}$, we have
\begin{eqnarray}
& \lim\limits_{n\to \infty} \, \int_{\R^{3}}\, \frac{u_{n}^{2}}{r^{2}}\, dx = 0 \label{s2.1}\\[0.2cm]
& \lim\limits_{n\to \infty} \, \int_{_{B_{R}(0)\times \R}}\, u_{n}^{2}\, dx = 0 \qquad \forall\, R>0 \label{s2.2}
\end{eqnarray}
Indeed by definition of $E(u,\omega,\ell)$ it immediately follows that 
$$e_{r,\ell} \ge e + \inf\limits_{(u,\omega) \in C_{h}}\, \frac 12 \int_{\R^{3}}\, \frac{\ell}{r^{2}} u^{2}\, dx$$
hence \eqref{s2.1} is proved. Let us now write for any $R>0$
$$
\left( \int_{_{B_{R}(0)\times \R}}\, u_{n}^{2}\, dx \right)^{2}= 4\pi^{2}\, \left( \int_{\R}\, \left( \int_{0}^{R}\,  r u_{n}^{2} \, dr \right) dz \right)^{2} \le 
$$
$$
\le 4\pi^{2}\, \left( \int_{\R} \int_{0}^{R}\,  r^{3} u_{n}^{2} \, dr dz \right) \, \left( \int_{\R} \int_{0}^{R}\,  \frac{u_{n}^{2}}{r} \, dr  dz\right) \le 4\pi^{2}\, R^{2}\, \|u_{n}\|_{L^{2}}^{2} \, \int_{\R^{3}}\, \frac{u_{n}^{2}}{r^{2}}\, dx 
$$ 
Hence \eqref{s2.2} follows from \eqref{s1} and \eqref{s2.1}. By \eqref{s2.2} it follows that $u_{n}$ converges to 0 weakly in $L^{2}$, which contradicts \eqref{converg} (see also Lemma \ref{nonvanish}).

\vskip 0.2cm
\noindent \textit{Step 2.} $\inf\limits_{{\mathcal H}=h}\, {\mathcal E} = \inf\limits_{(u,\omega) \in C_{h}}\, E(u,\omega,0)$

See \cite{bbbm}.

\vskip 0.2cm
\noindent \textit{Step 3.} $\inf\limits_{{\mathcal H}=h,\, {\mathbf M}={\mathbf m}}\, {\mathcal E} = \inf\limits_{{\mathcal H}=h}\, {\mathcal E}$

The inequality
$$
\inf\limits_{{\mathcal H}=h,\, {\mathbf M}={\mathbf m}}\, {\mathcal E} \ge \inf\limits_{{\mathcal H}=h}\, {\mathcal E}
$$
is immediate since on the left hand side the infimum is taken on a smaller set, since we constrain also on the manifold of fixed absolute value for the angular momentum. The opposite inequality is proved by constructing a sequence $\set{\psi_{n}}$ such that ${\mathcal H}(\psi_{n})=h$ and ${\mathbf M}(\psi_{n})={\mathbf m}$, which satisfies 
\begin{equation} \label{da-fare}
{\mathcal E}(\psi_{n}) \to \inf\limits_{(u,\omega) \in C_{h}}\, E(u,\omega,0) =\inf\limits_{{\mathcal H}=h}\, {\mathcal E}
\end{equation}
For any fixed $\eps>0$ let $U(x)$ and $v(x)$ be non-negative radially symmetric functions in $C^\infty_0(\R^3,\R^+)$, let $\omega <0$ and $\ell\in \Z$ which satisfy
\begin{eqnarray}
& {\mathcal H}(U(x) e^{-\imath \omega t}) = - \int_{\R^N}\, \omega U^2\, dx = h-\eps \label{succ1}  \\[0.2cm]
& E(U,\omega,0) = \inf\limits_{(u,\omega) \in C_{h-\eps}}\, E(u,\omega,0) +\eps \label{succ2} \\[0.2cm]
& {\mathbf M}(U(x) e^{-\imath \omega t}) = {\mathbf 0} \label{succ3} \\[0.2cm]
& {\mathcal H}(v(x) e^{\imath(\ell \theta(y) - \omega t)}) = - \int_{\R^N}\, \omega v^2\, dx = \eps \label{succ4} \\[0.2cm]
& \| v \|_{H^1} \le 2\sqrt{\frac{\eps}{|\omega|}} \label{succ5} \\[0.2cm] 
& {\mathbf M}(v(x) e^{\imath(\ell \theta(y) - \omega t)}) = \left(0,\, 0,\, -\ell\, \int_{\R^N}\, \omega v^2\, dx\right) = {\mathbf m} \label{succ6}
\end{eqnarray}
Recalling the notation $x=(y,z) \in \R^2 \times \R$, we now define
$$
\psi_n(x) = U(x) e^{-\imath \omega t} + v(y-y_n,z) e^{\imath(\ell \theta(y) - \omega t)}
$$
where $\set{y_n}\subset \R^2$ is a sequence satisfying $|y_n| \to \infty$. For $n$ big enough the supports of $U$ and $v(y-y_n,z)$ are disjoint, hence since the integrals in the expressions of charge and angular momentum are translation invariant, we obtain
$$
{\mathcal H}(\psi_{n})= {\mathcal H}(U(x) e^{-\imath \omega t}) + {\mathcal H}(v(x) e^{\imath(\ell \theta(y) - \omega t)}) = h
$$
from \eqref{succ1} and \eqref{succ4}, 
$$
{\mathbf M}(\psi_{n})= {\mathbf M}(U(x) e^{-\imath \omega t}) + {\mathbf M}(v(x) e^{\imath(\ell \theta(y) - \omega t)}) = {\mathbf m}
$$
from \eqref{succ3} and \eqref{succ6}. We now compute the energy of $\psi_n$. Again, since supports are disjoint, we can write
$$
\inf\limits_{(u,\omega) \in C_{h}}\, E(u,\omega,0) \le {\mathcal E}(\psi_n) = {\mathcal E}(U(x) e^{-\imath \omega t}) + {\mathcal E}(v(x) e^{\imath(\ell \theta(y) - \omega t)}) = 
$$
$$
= E(U,\omega,0) + E(v,\omega,0) + \frac 12\, \int_{\R^3}\, \frac{\ell^2}{r^2} v^2(y-y_n,z) dydz =
$$
$$
=  \inf\limits_{(u,\omega) \in C_{h-\eps}}\, E(u,\omega,0) +\eps + E(v,\omega,0) + \frac 12\, \int_{\R^3}\, \frac{\ell^2}{(r+|y_n|)^2} v^2(y,z) dydz \le
$$
$$
\le \inf\limits_{(u,\omega) \in C_{h-\eps}}\, E(u,\omega,0) +\eps + \nu\left(2\sqrt{\frac{\eps}{|\omega|}}\right) + \frac{\ell^2}{2}\, \frac{1}{|y_n|}\, \|v\|_{L^2}^2
$$
where $\nu(\cdot)$ is the modulus of continuity of the energy $E(u,\omega,0)$ with respect to the $H^1$ norm, and we have used \eqref{succ2} and \eqref{succ5}. 

Finally we remark that
$$
\inf\limits_{(u,\omega) \in C_{h-\eps}}\, E(u,\omega,0) = \inf\limits_{u\in H^1\setminus \set{0}}\, \left( \int\left[ \frac {1}{2}\left| \nabla u\right| ^{2}+W(u)\right] dx + \frac 12 \frac{(h-\eps)^2}{\int u^2 dx} \right) \le
$$
$$
\le \inf\limits_{u\in H^1\setminus \set{0}}\, \left( \int\left[ \frac {1}{2}\left| \nabla u\right| ^{2}+W(u)\right] dx + \frac 12 \frac{h^2}{\int u^2 dx} \right) = \inf\limits_{(u,\omega) \in C_{h}}\, E(u,\omega,0)
$$
We can conclude that there exists $n(\eps)$ for which
$$
\inf\limits_{(u,\omega) \in C_{h}}\, E(u,\omega,0) \le {\mathcal E}(\psi_{n(\eps)}) \le \inf\limits_{(u,\omega) \in C_{h}}\, E(u,\omega,0) +2\eps + \nu\left(2\sqrt{\frac{\eps}{|\omega|}}\right)
$$
and \eqref{da-fare} is proved.
\end{proof}

\section{Numerical approach to hylomorphic vortices for NKG} \label{sec:num-vort}

In the current section we consider for ease of presentation only the case $N=2$, so that $u(x)=u(r)$ is a radially symmetric function. Existence of hylomorphic vortices for $N=2$ is shown in \cite{befo09}. In particular, once introduced a straightforward method for the numerical construction of vortices, we present relevant simulations addressing their orbital stability. 

\subsection{Numerical construction of vortices} 

It has been shown in Section \ref{sec:ex-vort} that, for fixed $\ell\in \Z\setminus \{0\}$, vortices can be found as points of minimum for the energy $E(u,\omega,\ell)$
on the manifold $C_{h}$ of fixed charge $H(u,\omega) = h$.
Such a constrained minimization problem in two variables can be reformulated as an unconstrained minimization problem in one variable, as shown above in \eqref{j-def}. In particular, hereafter we can use $u\in \tilde H^{1}_{c}$ as independent variable and, letting 
\begin{equation} \label{omega_dalla_carica}
\omega \left( u,h \right) := - \frac{h}{\int u^{2}\, dx}\, ,
\end{equation}
numerically study the minimization problem for the hylomorphy ratio 
\begin{equation} \label{ganzetto}
\Lambda_{h} \left(u,\ell \right) := \frac{J_{h}(u,\ell)}{h} =
\frac{1}{h} \int \left[ \frac{1}{2}\left| \nabla u\right| ^{2}
+\frac{1}{2}\frac{\ell^2}{r^2}u^{2}
+W(u)\right] dx +\frac{h}{2\int u^{2}dx}.
\end{equation}
on $\tilde H^{1}_{c}$. To this purpose, we consider the evolutionary problem (which generalizes the one treated in \cite{bel})
\begin{equation} \label{flusso_parabolico}
\left\{
\begin{array}{rclcl}
    \partial_\tau u (r,\tau) & = &  - h \, \mathrm{d} \Lambda _{h} =
    \Delta u
    - W^\prime(u)
    + \left(\omega^2 - \displaystyle \frac{\ell^2}{r^2}\right)u
    & \mathrm{in} & (0,\tilde{r})\, \times \R^+ \\
    u(r,\tau) & = & 0
    & \mathrm{on} & \set{r = 0} \times \R^+ \\
    u(r,\tau) & = & 0
    & \mathrm{on} & \set{r = \tilde{r}} \times \R^+ \\
\end{array}
\right.
\end{equation}
in which $\tau$ represents a pseudo-time,
$\tilde{r}$ denotes a chosen upper bound for the $r-$domain (discussed below) and
$\omega=\omega(u,h)$ as in (\ref{omega_dalla_carica}).
The evolution of $u$ according to (\ref{flusso_parabolico}) is a gradient flow and therefore
a non-increasing trend of $\Lambda_{h} (u(r,\tau))$ versus $\tau$ is obtained, well suited as
the sought minimization process.

The problem (\ref{flusso_parabolico}) was then discretized by a classical line method; in particular,
2nd order and 1st order accurate finite differences were respectively used for space and time
discretization. Moreover,
$\tilde{r}$ was chosen large enough to suitably contain the vortex profile support,
and the chosen charge $h$ was directly enforced at each time level by
evaluating the frequency $\omega^n$ (superscript $n=0,1,2,\dots$ hereafter denoting
the $n$-th time level) from the corresponding
numerical solution through the discrete counterpart of
(\ref{omega_dalla_carica}).
Finally, time-advancing was stopped when
$|\omega^{n+1}-\omega^n|/\omega^n < e_\omega$, where $e_\omega$ represents a predefined
threshold (a relative error on $\Lambda_h$ was considered as well).
It is worth remarking that:
(i) the proposed method managed to efficiently converge by starting from several initial guesses;
(ii) for many of the carried out numerical experiments it was necessary to adopt a very severe
threshold, e.g. $e_\omega=10^{-11}$, in order to obtain discretization-independent
results (this point suggested the possible presence of rather flat regions around the sought minimum, somehow
challenging the achievement of the ``true'' radial profile).
\\
Let us now consider, for the sake of illustration, the following potential: 
\begin{equation}
W \left( \psi \right) =
\left| \psi \right| -\log \left( 1+\left| \psi \right|\right)
= \frac 1 2\, |\psi|^{2} - \frac 1 3\, |\psi|^{3} + o(|\psi|^{3}),
\label{ex_W_gamma}
\end{equation}
for which NKG specializes to
$$
\square \psi +\frac{\psi }{1+\left| \psi \right| }=0.
$$

\noindent
Figure \ref{fig_u_vs_h_l} shows some
radial profiles $u(r)$ of two-dimensional (2D) vortices
obtained by minimizing the functional
(\ref{ganzetto}) coupled with (\ref{ex_W_gamma}), through the
procedure described above.
It should be noticed that: (i) for a given $\ell$, the peak of the $u$ profile significantly
decreases by decreasing the charge $h$, while the associated support is less affected;
(ii) for a given charge $h$, the peak of the $u$ profile smoothly 
decreases and drifts away from the origin by increasing $\ell$.
Moreover, Figure \ref{fig_trend_om_lambda_vs_h_l} shows characteristic
trends for both frequency and hylomorphy ratio, still versus $h$ and $\ell$.
It should be noticed how both $|\omega|$ and $\Lambda$ tend to $1$ for $h\to0$,
as well as for increasing values of $\ell$.
Finally, once obtained $u(r)$, it is straightforward to build vortices through (\ref{ansatz});
two examples are shown in Figure \ref{fig_2Dvort} (a grid is added for ease of rendering).

\begin{figure}
    \begin{center}
    \subfigure[]
    {\includegraphics[width=6.0cm,height=4.5cm]{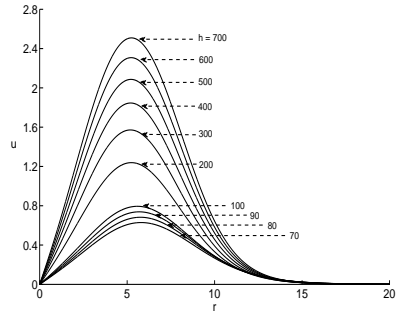}}
    \subfigure[]
    {\includegraphics[width=6.0cm,height=4.5cm]{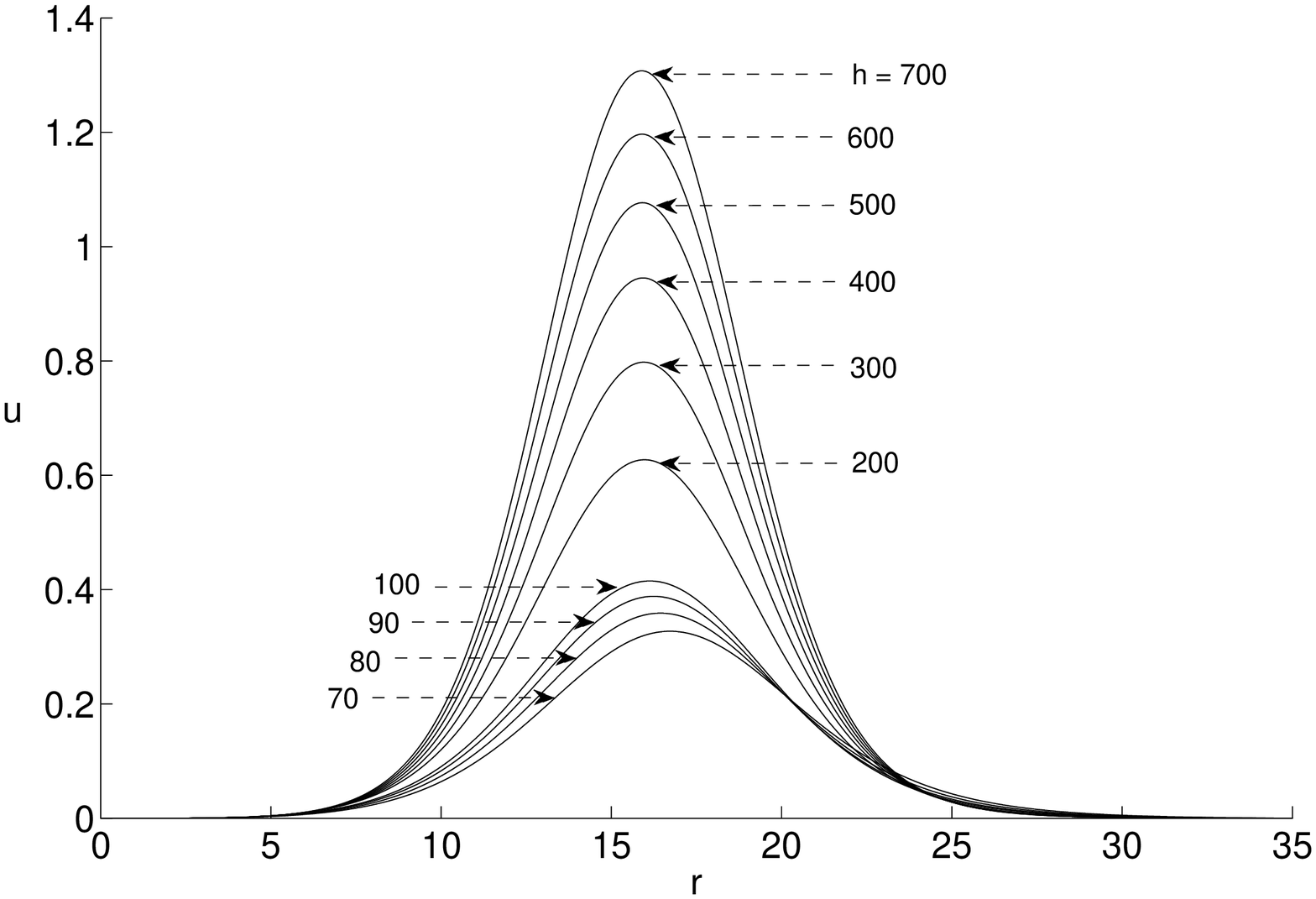}}
    \subfigure[]
    {\includegraphics[width=6.0cm,height=4.5cm]{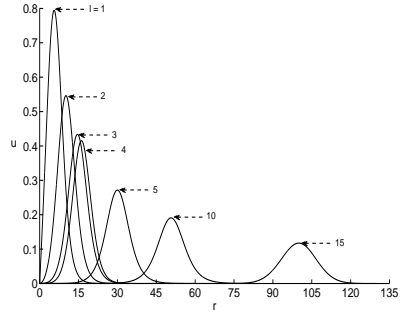}}
    \subfigure[]
    {\includegraphics[width=6.0cm,height=4.5cm]{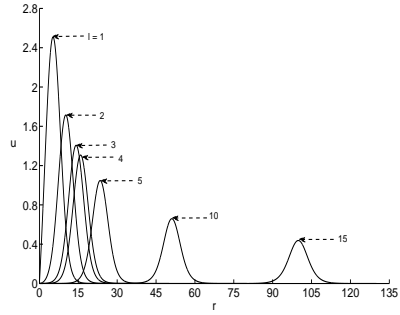}}
    \caption{Radial profiles $u(r)$ of 2D vortices:
             for given values of $h$,
             having fixed $\ell=1$ (a) and $\ell=4$ (b);
             for given values of $\ell$,
             having fixed $h=100$ (c) and $h=700$ (d).
             $W$ defined as in (\ref{ex_W_gamma}).}
    \label{fig_u_vs_h_l}
    \end{center}
\end{figure}

\begin{figure}
    \begin{center}
    \subfigure[]
    {\includegraphics[width=6.0cm,height=4.5cm]{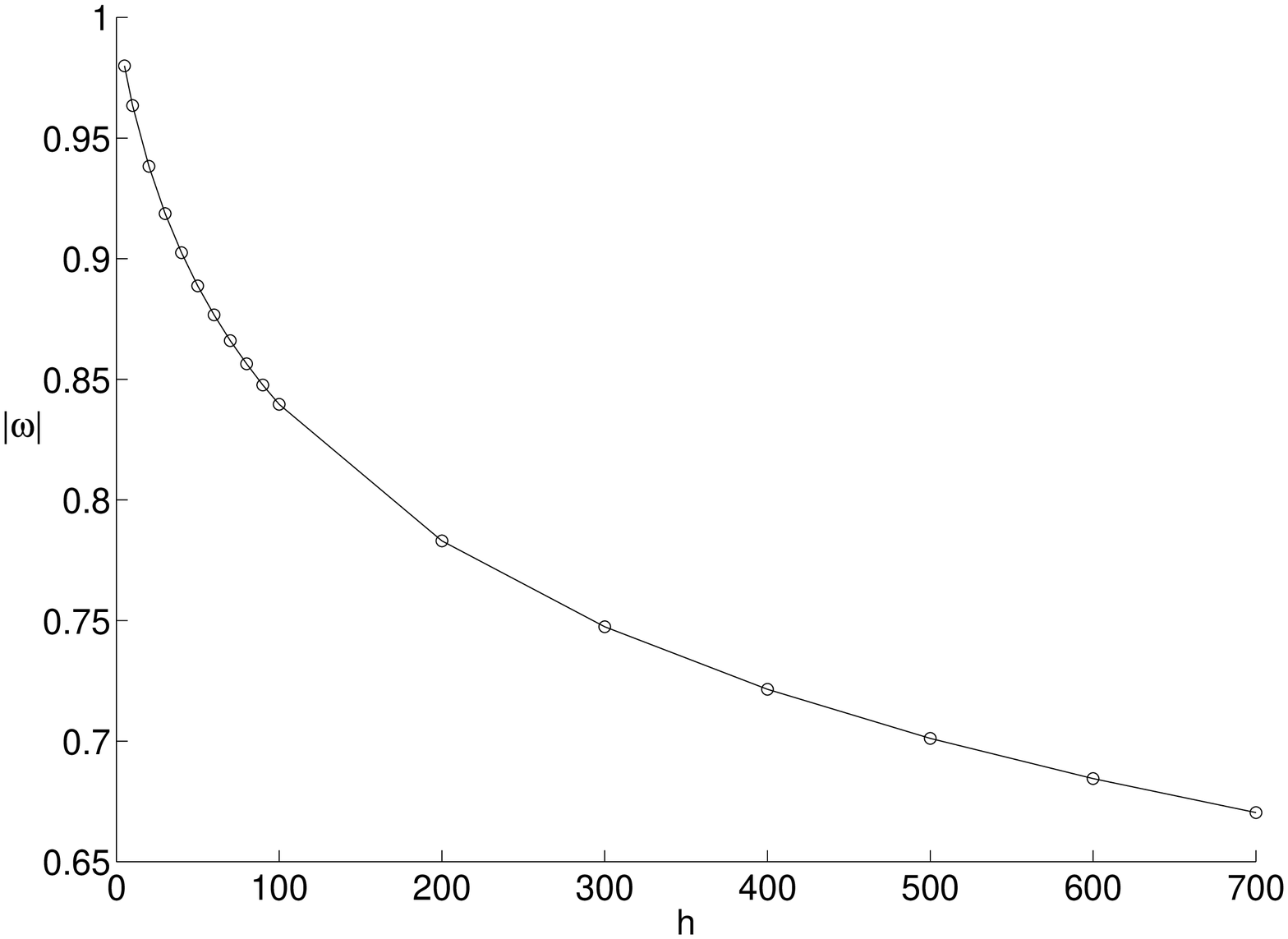}}
    \subfigure[]
    {\includegraphics[width=6.0cm,height=4.5cm]{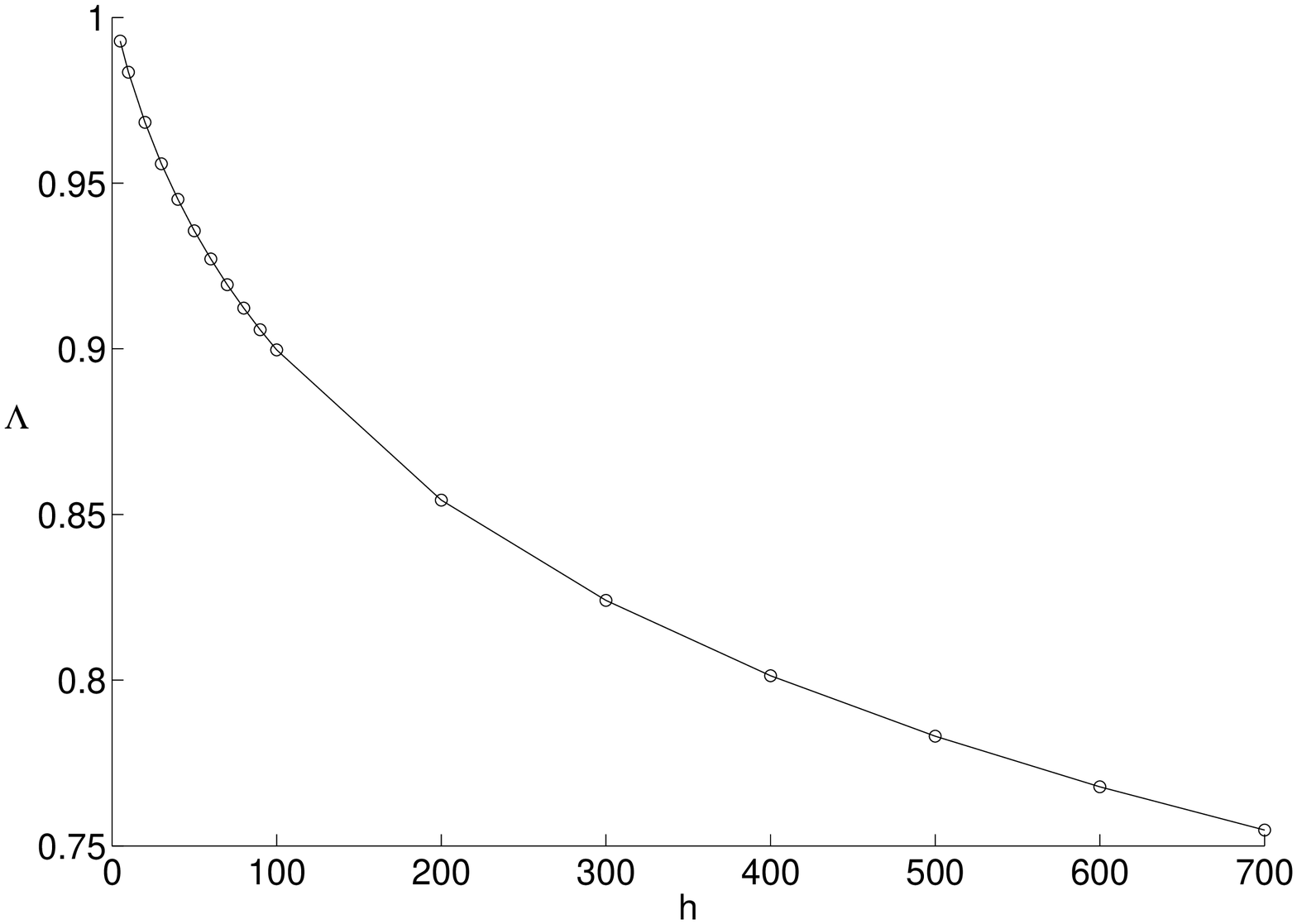}}
    \subfigure[]
    {\includegraphics[width=6.0cm,height=4.5cm]{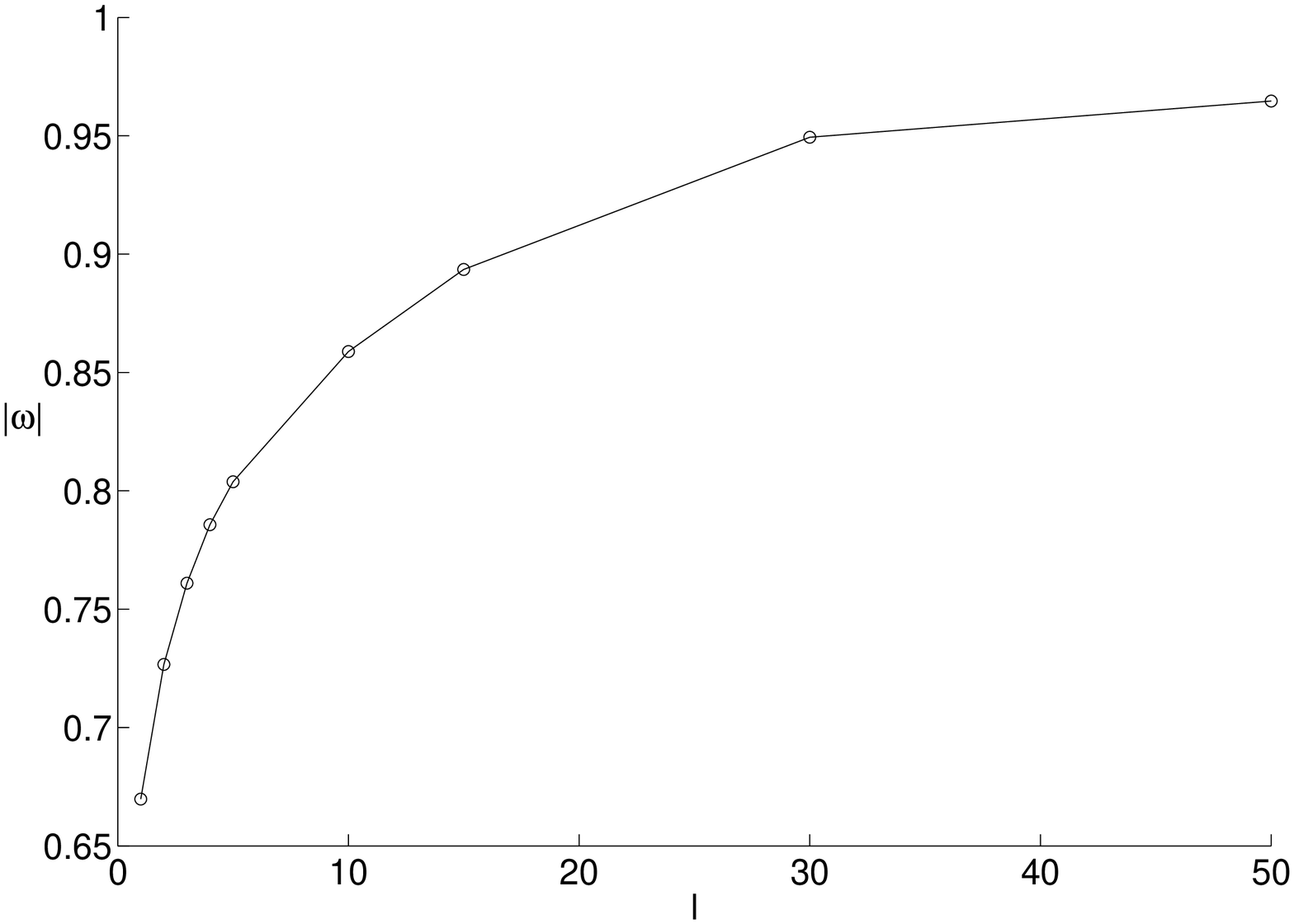}}
    \subfigure[]
    {\includegraphics[width=6.0cm,height=4.5cm]{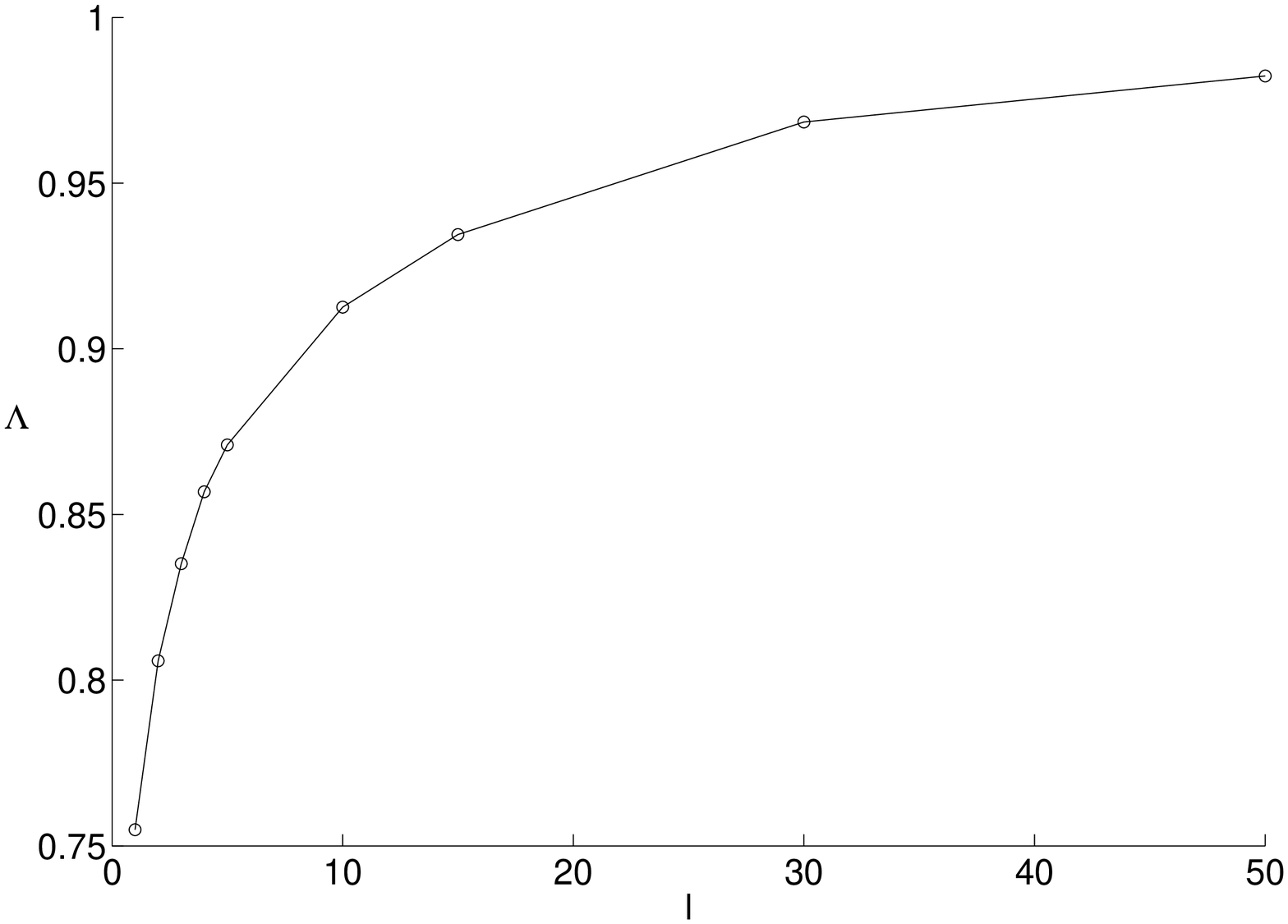}}
    \caption{Frequency (a,c) and hylomorphy ratio (b,d) of 2D vortices:
             as a function of $h$, having fixed $\ell=1$ (a,b);
             as a function of $\ell$, having fixed $h=700$ (c,d).
             Circles in (a,b) are associated with the following values of $h$:
             $\{5,10,20,30,40,50,60,70,80,90,100,200,300,400,500,600,700\}$;
             circles in (c,d) are associated with the following values of $\ell$:
             $\{1,2,3,4,5,10,15,30,50\}$.
             $W$ defined as in (\ref{ex_W_gamma}).}
    \label{fig_trend_om_lambda_vs_h_l}
    \end{center}
\end{figure}

\begin{figure}
    \begin{center}
    \subfigure[]
    {\includegraphics[width=8.0cm,height=6.0cm]{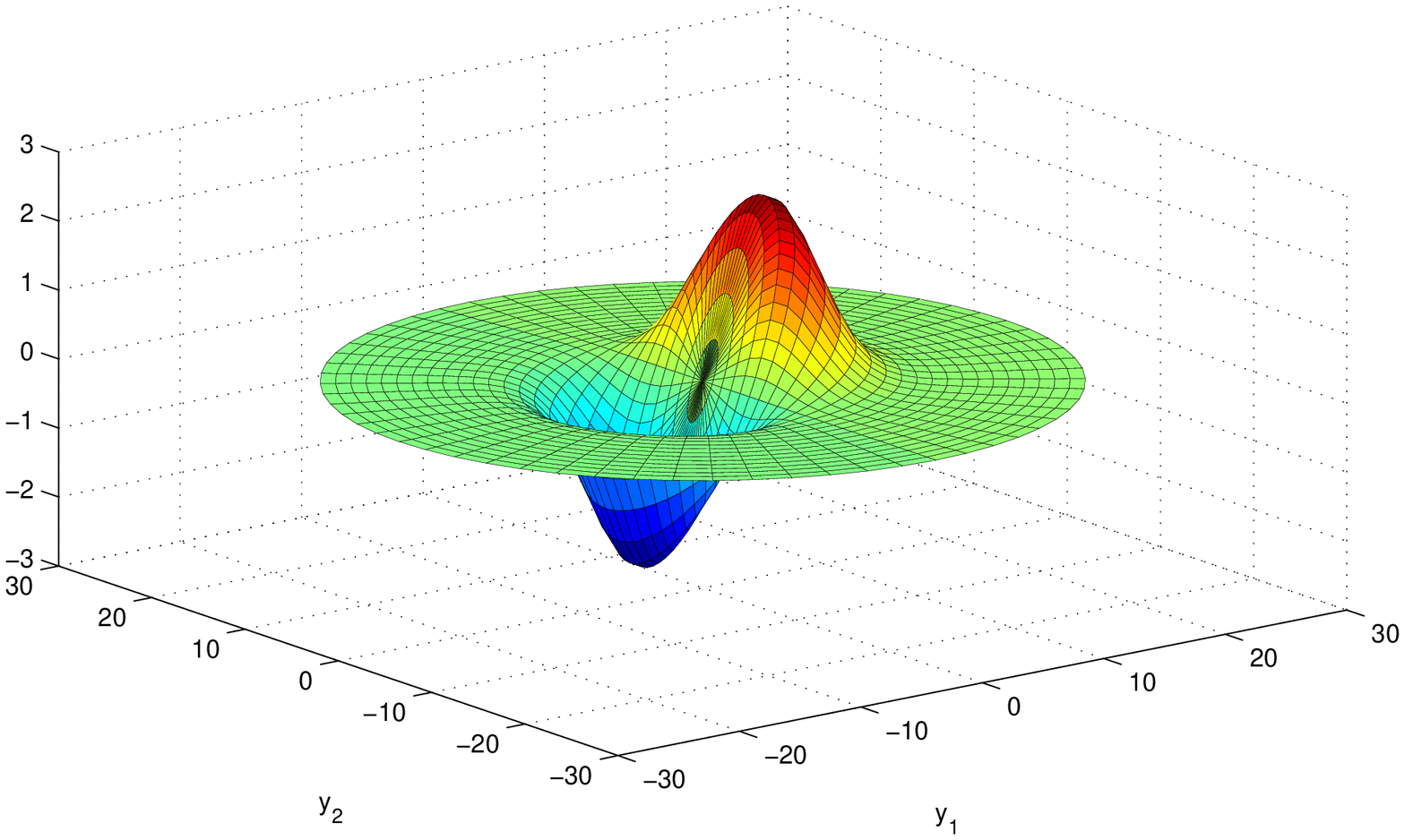}}
    \subfigure[]
    {\includegraphics[width=8.0cm,height=6.0cm]{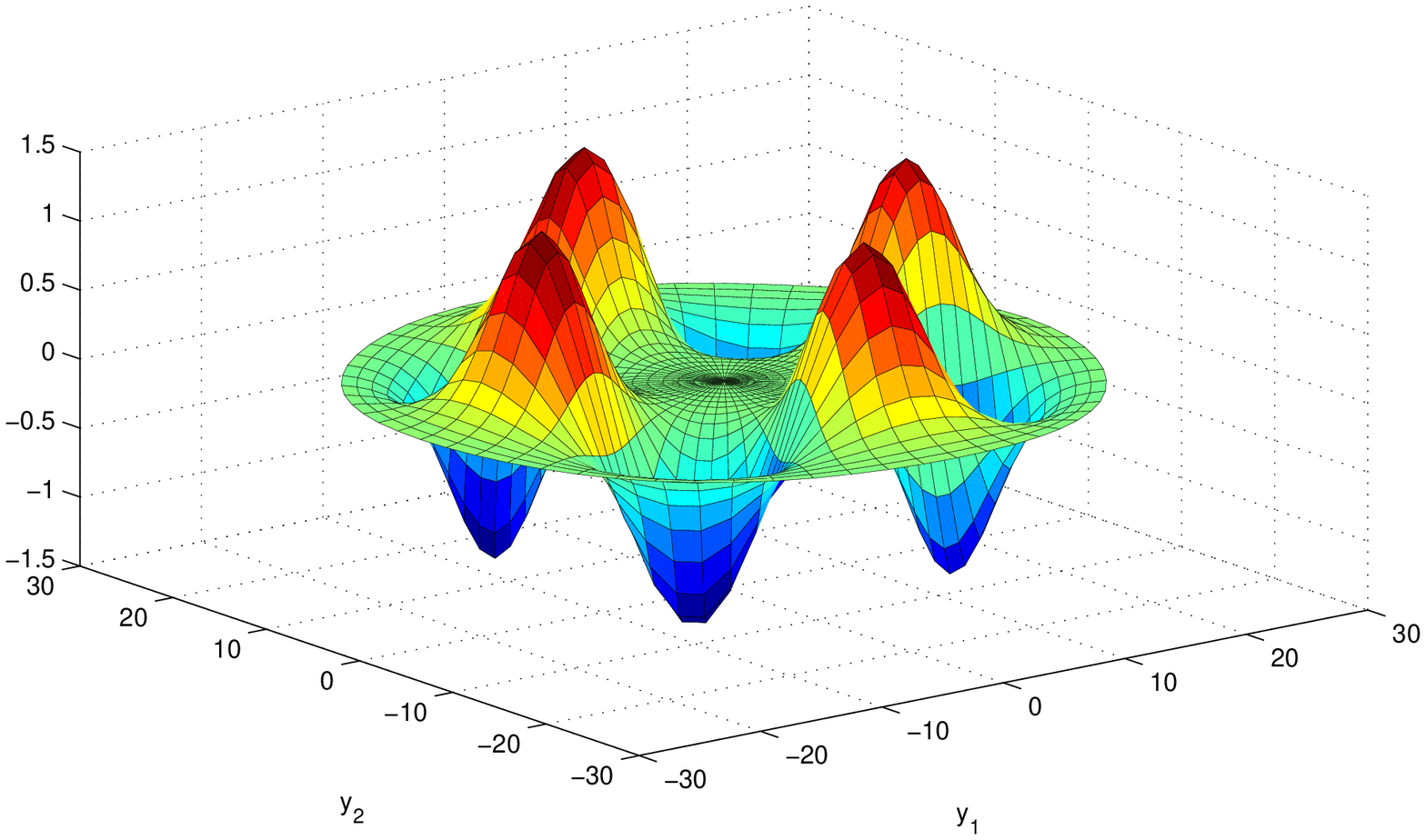}}
    \caption{Surface plot of $\mathrm{Re}\left(\psi\right)$ for a 2D vortex with: $h=700$, $\ell=1$ (a);
             $h=700$, $\ell=4$ (b) (corresponding radial profiles are
             shown in Figure \ref{fig_u_vs_h_l}); 
             $W$ defined as in (\ref{ex_W_gamma}).}.
    \label{fig_2Dvort}
    \end{center}
\end{figure}

\subsection{Numerical experiments on the stability of vortices}

Numerical approximation of NKG has long been studied (see e.g.
\cite{Vu-Quoc1993}, \cite{Duncan1997} and \cite{wang2005}), in order to also investigate nonlinear phenomena
like the dynamics/interaction of solitary waves or other coherent structures.
In light of the relevant literature, we decided to preliminarily assess the suitability of a line method
approach based on centered finite-differencing for Laplacian
discretization and leap-frog time-advancing, by virtue of its relatively simple coding.
In particular, we firstly considered the time-evolution of a (non-rotating) hylomorphic soliton for which orbital stability is proved in \cite{bel},
on a square domain with periodic boundary conditions (2D-torus).
The adopted numerical scheme is symplectic, 
thus being 
able to accurately capture the considered dynamics over a long time-interval
(see e.g. \cite{bridges2006}, \cite{Duncan1997} and \cite{skeel1997}).

More in detail, we firstly chose $h=500$ and we obtained the radial profile and the frequency of a non-rotating soliton as in \cite{bel} (i.e. by a minimization strategy like (\ref{flusso_parabolico}), yet with $\ell=0$),
with $W$ defined as in (\ref{ex_W_gamma}).
We then introduced a square grid with spacing $\delta x$ (and side length large enough to properly contain the soliton support),
and we endowed the aforementioned soliton with speed $v=0.5$ along a grid axis, by Lorentz transform. 
Let us denote by ($\psi^{th}(x,t)$,$\psi_t^{th}(x,t)$) the obtained (translating) theoretical soliton.
Furthermore, once introduced the time-step $\delta t$,
we defined the initial conditions as $\psi^0:=\tilde{\psi}^{th}(x,0)$ and $\psi_t^{-1/2}:=\tilde{\psi}_t^{th}(x,-\delta t/2)$), where
$\tilde{\cdot}$ stands for sampling on the computational grid, and
the $\delta t/2$ time-shift accounts for time-staggering of the
leap-frog scheme \cite{skeel1997}. Therefore, after $n$ time-steps the adopted scheme provides $\psi^n \approx \tilde{\psi}^{th}(x,n \, \delta t)$ and $\psi_t^{n-1/2} \approx \tilde{\psi}_t^{th}(x,\left(n-1/2\right)\, \delta t)$.
Time-staggering clearly affects the numerical approximation of those entities which simultaneously involve $\psi$ and $\psi_t$, like e.g. all NKG first integrals; however, corresponding error can be kept contained by adopting small time-steps, as required by stability constraints (see below).
In light of this aspect and with the main aim of investigating orbital stability, a ``staggered'' discrete counterpart of (\ref{os_continuo}) was introduced, namely the following orbital stability norm:
\begin{equation}
\delta_{OS}^n := \frac{\| \psi^n         - \tilde{\psi}^{th}(x,n \, \delta t) \|_{H^{1}}}
                      {\| \tilde{\psi}^{th}(x,0) \|_{H^{1}}} +
                 \frac{\| \psi_t^{n-1/2} - \tilde{\psi}_t^{th}(x,\left(n-1/2\right)\, \delta t) \|_{L^{2}}}
                      {\| \tilde{\psi}_t^{th}(x,-\delta t/2) \|_{L^{2}}}, \quad n = 0,1,\ldots,
\label{OSN_discreta}
\end{equation}
where discrete $H^{1}$ and $L^{2}$ norms are tacitly understood (they were computed by classical fourth-order Simpson quadrature),
and the denominators (involving time-conserved entities) simply act as scaling factors.
It should be noticed that the adopted initial conditions enforce $\delta_{OS}^0=0$; subsequent evolution represents the main
asset of the considered numerical investigations. Moreover, it is worth remarking that numerical approximations of NKG first integrals
were defined similarly to (\ref{OSN_discreta}), yet corresponding expressions are here omitted for brevity.
Formally, the considered numerical scheme introduces a truncation error $O(\delta x^2 + \delta t^2)$. However, stability constraints
affect the adopted explicit time-advancing (see e.g. \cite{Vu-Quoc1993} and \cite{skeel1997} for some model problems),
so that a ratio $\delta t/\delta x = 1/10$ was adopted for the numerical experiments at hand (it was also
checked that such a ratio provided time-discretization-independent results).
Hence, it was space discretization to
directly modulate the leading term $O(\delta x^2)$ of truncation error.
Consequently, grid coarsening/refinement was regarded to
as a mean for implicitly tuning a perturbation on the theoretical soliton profile which was
exploited to seed the initial conditions.
More in detail, several numerical simulations were carried out by varying $\delta x$; corresponding $\delta_{OS}^n$ trends are
reported in Figure \ref{fig_nos_qball}.

\begin{figure}[h!]
    \begin{center}
    {\includegraphics[width=8.0cm,height=6cm]{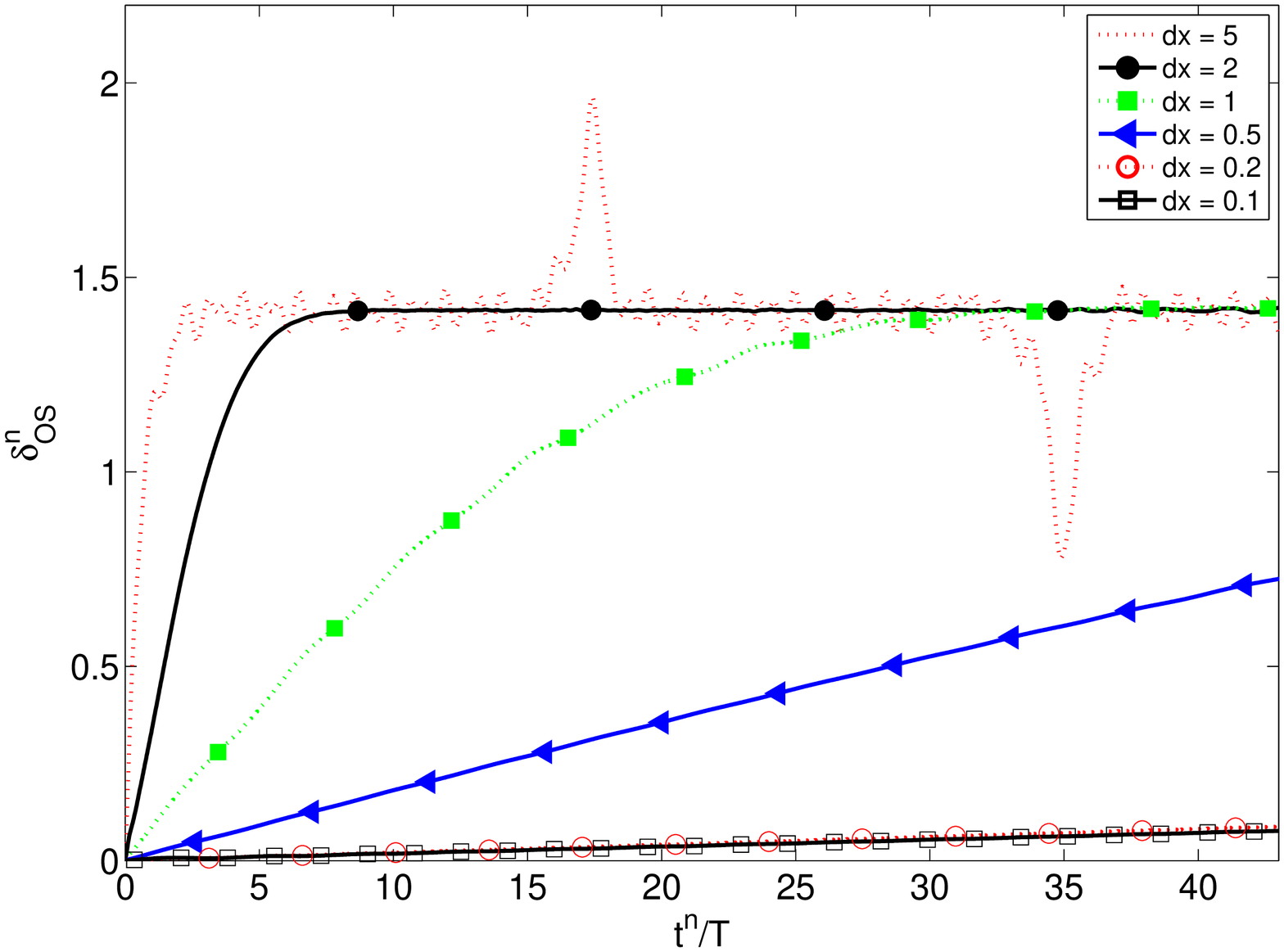}}
    \caption{Exemplificative trends of $\delta_{OS}^n$ versus time $t^n = n\, \delta t$ for
             an orbitally-stable hylomorphic soliton ($\ell=0$, non-rotating) with $h=500$;
             $W$ defined as in (\ref{ex_W_gamma}).
             Time non-dimensionalised through the soliton period $T$.
             Each curve is associated with a grid; corresponding characteristic size $\delta x$ was varied in
             $\{5,2,1,0.5,0.2,0.1\}$.
             Non-rotating solitons were preliminarily simulated,
             in order to assess the suitability of the proposed numerical approach.}
    \label{fig_nos_qball}
    \end{center}
\end{figure}

\noindent
Conservation of NKG first integrals  was systematically monitored for all the 
considered simulations. For instance, energy variation relative to its time-average was in the range
$(-8,8)\cdot 10^{-2}$, $(-1,1)\cdot 10^{-3}$ and $(-2,2)\cdot 10^{-6}$, when respectively adopting $\delta x=5$, $1$ and $0.1$.
Corresponding relative variations for charge were below $10^{-12}$ in absolute value.
These results strongly support the remarkable features of the adopted symplectic integration scheme,
while also establishing sound foundations for the interpretation of the obtained $\delta_{OS}^n$ trends.
On regard,
a rough discretization like e.g. the one associated with $\delta x = 5$ introduced a disrupting perturbation
on the seeded profile: the numerical solution soon lagged behind the moving theoretical one. This, in turn, only periodically allowed for support
overlapping: for almost all times, numerical and theoretical solutions occupied disjoint regions of the computational grid, and this justifies the
asymptotic value around $\sqrt{2}$ in Figure \ref{fig_nos_qball}. Moreover, as grid spacing was reduced, both the initial solution was better sampled
(with respect to the theoretical shape) and the truncation error decreased.
This, in turn, incrementally led to
grid-independent numerical results (see the curves clustering
in Figure \ref{fig_nos_qball}, where lower values of $\delta x$ are not shown for ease of readability),
which also exhibit an orbitally stable character (the weakly increasing trend is merely
due to a cumulative effect of discretization errors over time).
The obtained numerical results, which are fully consistent with the underlying theoretical framework, encouraged to extend to the vortex case the proposed approach, as well as the idea of interpreting space discretization as a mean for issuing tailored
perturbations with respect to theoretical profile of the considered solitons.

The same numerical approach was therefore applied to rotating solitons, by firstly considering 
a vortex with $\ell=4$ and charge $h=500$.
In particular, a polar grid was introduced and the characteristic discretization size $\delta r$ along the radial direction was
chosen as the leading term for truncation error. This was achieved by systematically adopting a finer discretization along the
circumpherential direction, and by choosing $\delta t/\delta r = 1/50$ in order to ensure stability (indeed, the ratio $1/10$ previously adopted for non-rotating solitons turned out to be insufficient in the present case).
The initial conditions were defined as above, and the simulation was advanced by monitoring first integrals
and the orbital stability norm defined by (\ref{OSN_discreta}). In particular, several numerical experiments were carried out by varying $\delta r$; corresponding $\delta_{OS}^n$ trends are
reported in Figure \ref{fig_nos_vortici}~(b).

\begin{figure}[h!]
    \begin{center}
    \subfigure[]
    {\includegraphics[width=8.0cm,height=6cm]{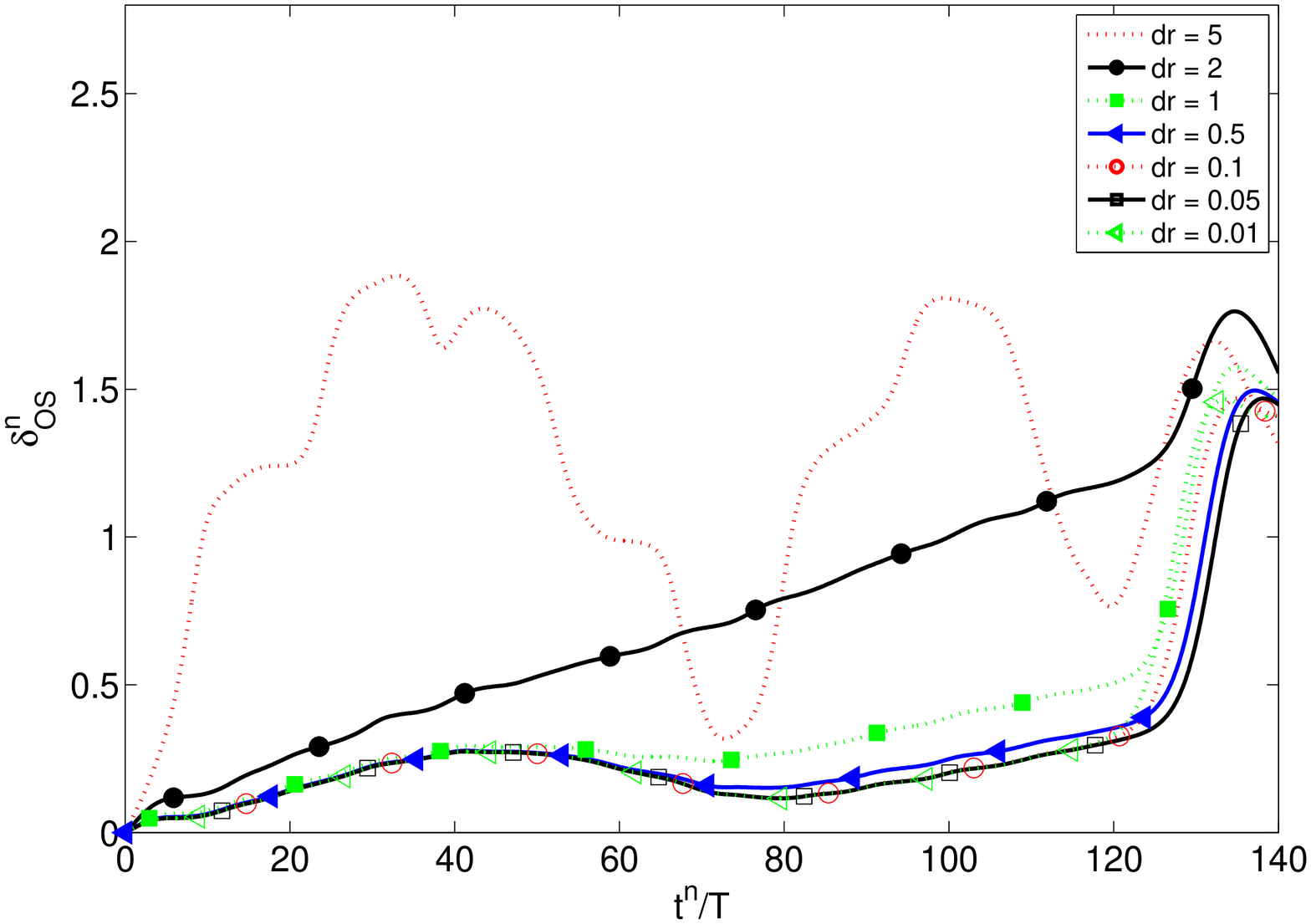}}
    \subfigure[]
    {\includegraphics[width=8.0cm,height=6cm]{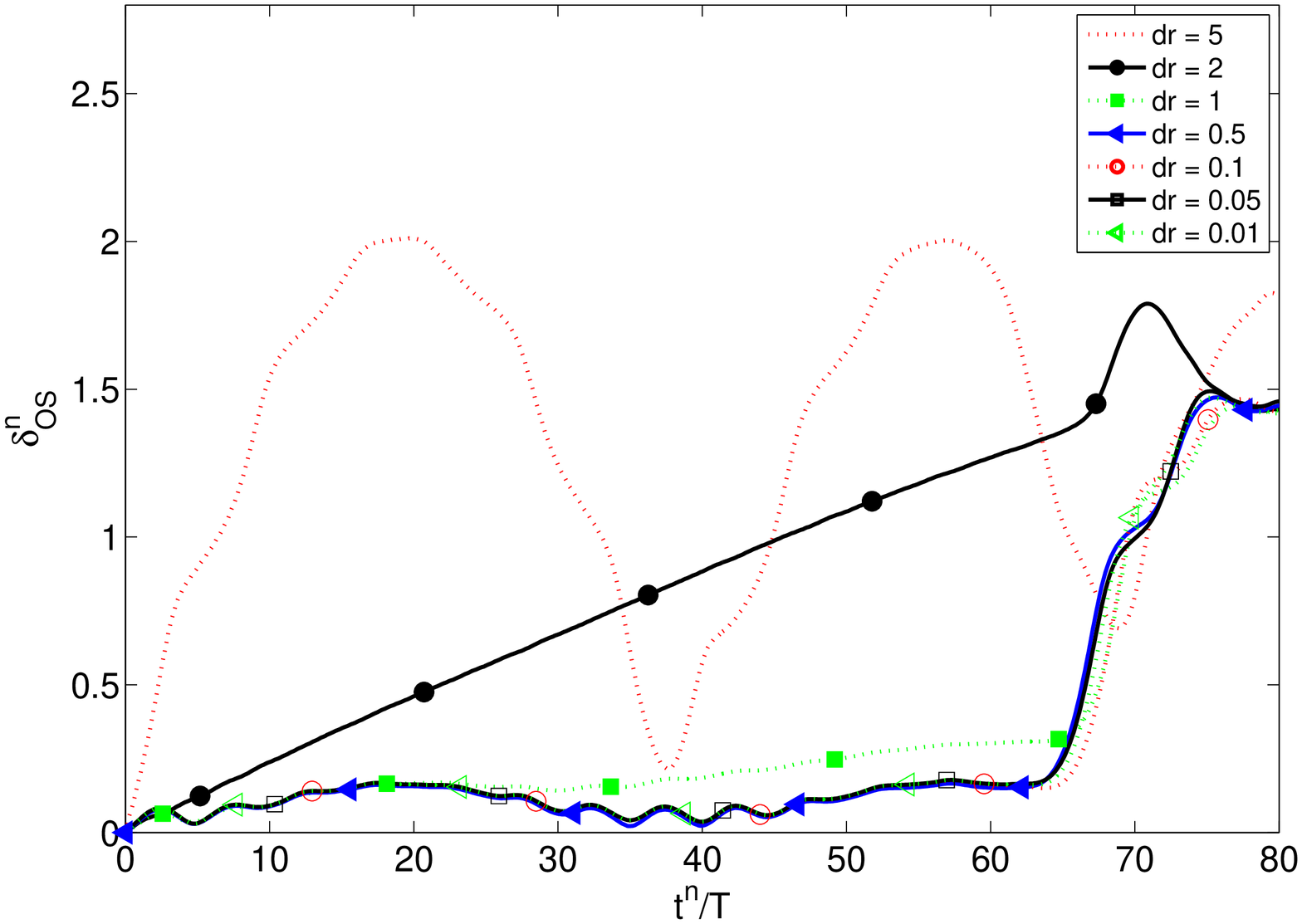}}
    \caption{Exemplificative trends of $\delta_{OS}^n$ versus time $t^n = n\, \delta t$ for
             hylomorphic vortices with $\ell=4$,
             having fixed $h=100$ (a) and $h=500$ (b); $W$ defined as in (\ref{ex_W_gamma}).
             Time non-dimensionalised through the soliton period $T$.
             Each curve is associated with a grid; corresponding characteristic size $\delta r$ was varied in
             $\{5,2,1,0.5,0.1,0.05,0.01\}$. The obtained trends are remarkably different from the ones in
             Figure \ref{fig_nos_qball}.}
    \label{fig_nos_vortici}
    \end{center}
\end{figure}

\noindent 
Conservation of NKG first integrals was systematically monitored for all the simulated vortices.
For instance, energy variation relative to its time-averaged value was in the range
$(-1,8)\cdot 10^{-2}$, $(0,5)\cdot 10^{-3}$ and $(-1,1)\cdot 10^{-5}$,
when respectively adopting $\delta r=5$, $1$ and $0.1$.
Corresponding ranges for charge were
$(-12,1)\cdot 10^{-2}$, $(-1,1)\cdot 10^{-6}$ and $(-2,2)\cdot 10^{-10}$;
corresponding ranges for the angular momentum were
$(-15,5)\cdot 10^{-2}$, $(-1,1)\cdot 10^{-6}$ and $(-2,2)\cdot 10^{-8}$.
These results further supported the adoption of the above described numerical scheme, and gave us
confidence in the obtained $\delta_{OS}^n$ trends.
On regard, no stable dynamics was observed: vortex rupture occurred for all the considered test-cases,
at the latest nearly $t/T = 63$ on the most refined grids (such a threshold is clearly visible in the figure).
A rough discretization like the one associated with $\delta r = 5$, for instance, immediately caused a phase lag
between numerical solution and theoretical vortex, which in turn produced the corresponding periodic $\delta_{OS}^n$ trend
in Figure \ref{fig_nos_vortici}~(b). 
Moreover, perturbation reduction (as induced by grid refinement) did not prevent vortex rupture; most importantly,
not even it affected rupture time (e.g. by introducing incremental delays).
This aspect is most evident in the figure: rupture time was not significantly affected by a reduction of two
orders of magnitude in the grid characteristic size (say, from $\delta r=1$ down to $0.01$). 
Such results are completely different from the ones observed for hylomorphic solitons, i.e. in the orbitally stable case.
Nonetheless, the aforementioned numerical experiments were repeated by considering a vortex with $\ell=4$ and charge $h=100$,
in order to possibly corroborate the preliminarily obtained results.
Corresponding trends for $\delta_{OS}^n$ are reported in Figure \ref{fig_nos_vortici}~(a):
they fully confirm those observed for $h=500$. This point, together with the fact that rupture occurred
nearly $t/T = 125$ for such a smaller charge value, incidentally suggests that no trivial correlation between charge and
orbital stability might be inferred at the present research stage.
Overall, the carried out numerical simulations suggest that vortices may be unstable.

\section{Acknowledgements}
The authors would like to thank Francesca Guerra for her support
in carrying out the numerical simulations. J.B., V.B. and C.B. are sponsored by project MIUR - PRIN2009 ``Variational and topological methods in the study of nonlinear phenomena", Italy. C.B. is sponsored also by the ``Distinguished Scientist Fellowship Program (DSFP)", King Saud University, Riyadh, Saudi Arabia.


\bigskip


\begin{thebibliography}{99}
\bibitem{bad} \textsc{Badiale M., Benci V., Rolando S.,} \textit{Three
dimensional vortices in the nonlinear wave equation,} Boll. Unione Mat. Ital. (9), \textbf{2} (2009), 105--134
\bibitem{BBR07} \textsc{Badiale M., Benci V., Rolando S.,} \textit{A nonlinear elliptic equation with singular potential and application to nonlinear equations,} J. Eur. Math. Soc \textbf{9} (2007), 355--381

\bibitem{BR12} \textsc{Badiale M., Rolando S.,} \textit{A note on vortices with prescribed charge}, Adv. Nonlinear Stud., \textbf{12} (2012), 703--716

\bibitem{bbbm} \textsc{Bellazzini J., Benci V., Bonanno C., Micheletti A.M.,}
\textit{Solitons for the nonlinear Klein-Gordon equation,} Adv. Nonlinear Stud., \textbf{10} (2010), 481--499

\bibitem{bel} \textsc{Bellazzini J., Benci V., Bonanno C., Sinibaldi E.,} \textit{Hylomorphic solitons in the nonlinear Klein-Gordon equation,} Dyn. Partial Differ. Equ., \textbf{6} (2009), 311--334

\bibitem{bebo} \textsc{Bellazzini J., Bonanno C.,} \textit{Nonlinear Schr\"{o}dinger equations with strongly singular potentials,} Proc. Roy. Soc. Edinburgh Sect. A, \textbf{140} (2010), 707--721

\bibitem{be09} \textsc{Benci V.,} \textit{Hylomorphic solitons,} Milan J. Math., \textbf{77} (2009), 271--332

\bibitem{befo09} \textsc{Benci V., Fortunato D.,} \textit{Existence of hylomorphic solitary waves in Klein-Gordon and in Klein-Gordon-Maxwell equations,} Atti Accad. Naz. Lincei Cl. Sci. Fis. Mat. Natur. Rend. Lincei (9) Mat. Appl., \textbf{20} (2009), 243--279

\bibitem{Be-Visc} \textsc{Benci V., Visciglia N.,} \textit{Solitary
waves with non vanishing angular momentum}, Adv. Nonlinear Stud., \textbf{3}
(2003), 151--160

\bibitem{Beres-Lions} \textsc{Berestycki H., Lions P.L.,} \textit{Nonlinear scalar field equations, I - Existence of a ground state}, Arch. Rational Mech. Anal., \textbf{82} (1983), 313--345

\bibitem{bona} \textsc{Bonanno C.,} \textit{Existence and multiplicity of stable bound states for the nonlinear Klein-Gordon equation}, Nonlinear Anal., \textbf{72} (2010), 2031--2046

\bibitem{bridges2006} \textsc{Bridges T.J., Reich S.} \textit{Numerical methods for Hamiltonian PDEs,} J. Phys. A: Math. Gen., \textbf{39} (2006), 5287--5320

\bibitem{caru} \textsc{Campanelli L., Ruggieri M.,} \textit{Spinning supersymmetric Q-balls}, Phys. Rev. D, \textbf{80} (2009), 036006

\bibitem{Coleman78} \textsc{Coleman S., Glaser V., Martin A.,} \textit{Action minima among solutions to a class of Euclidean scalar field equation,} Comm. Math. Phys, \textbf{58} (1978), 211--221

\bibitem{Coleman86} \textsc{Coleman S.,} \textit{Q-Balls}, Nuclear Phys. B, \textbf{262} (1985), 263--283; erratum: \textbf{269} (1986), 744--745

\bibitem{Duncan1997} \textsc{Duncan D.B.,} \textit{Finite difference approximations of the nonlinear Klein-Gordon equation,} SIAM J. Numer. Anal., \textbf{34} (1997), 1742--1760

\bibitem{el} \textsc{Esteban M., Lions P.L.,} \textit{A compactness lemma,} Nonlinear Anal., \textbf{7} (1983), 381--385

\bibitem{gss} \textsc{Grillakis M., Shatah J., Strauss W.,} \textit{Stability theory of solitary waves in the presence of symmetry, I}, J. Funct. Anal., \textbf{74} (1987), 160--197

\bibitem{Kim93} \textsc{Kim C., Kim S., Kim Y.,} \textit{Global nontopological vortices}, Phys. Rev. D, \textbf{47} (1985), 5434--5443

\bibitem{rosen68} \textsc{Rosen G.}, \textit{Particlelike solutions to
nonlinear complex scalar field theories with positive-definite energy
densities}, J. Math. Phys., \textbf{9} (1968), 996--998

\bibitem{skeel1997} \textsc{Skeel R.D., Zhang G., Schlick T.} \textit{A family of symplectic integrators: stability, accuracy, and molecular dynamics applications,} SIAM J. Sci. Comput., \textbf{18} (1997), 203--222

\bibitem{strauss} \textsc{Strauss W.A.}, \textit{Existence of solitary waves
in higher dimensions,} Comm. Math. Phys., \textbf{55} (1977), 149--162

\bibitem{VW} \textsc{Volkov M.S., W\"{o}hnert E.,} \textit{Spinning Q-balls,} Phys. Rev. D, \textbf{66} (2002), 085003

\bibitem{Vu-Quoc1993} \textsc{Vu-Quoc L., Li S.,} \textit{Invariant-conserving finite difference algorithms for the nonlinear Klein-Gordon equation,} Comp. Meth. Appl. Mech. Eng., \textbf{107} (1993), 341--391

\bibitem{wang2005} \textsc{Wang Y., Wang B.} \textit{High-order multi-symplectic schemes for the nonlinear Klein-Gordon equation,} Appl. Math. Comput., \textbf{166} (2005), 608--632

\end{thebibliography}
\end{document}